\documentclass[11pt]{amsart}
\usepackage{fullpage}
\usepackage{amsmath,amsthm,amssymb,pifont}
\usepackage{setspace} 
\usepackage{xcolor}
\usepackage{thmtools}
\usepackage{tikz}
\usepackage{pgfplots}
\pgfplotsset{compat=newest}
\usetikzlibrary{shapes.geometric}
\usetikzlibrary{decorations.pathreplacing}
\usetikzlibrary{positioning}

\usepackage{comment}
\usepackage{tcolorbox}
\usepackage[colorlinks=true, pdfstartview=FitV, linkcolor=blue, citecolor=black, urlcolor=black,filecolor=magenta]{hyperref}
\usepackage{subfig}
\usepackage{graphicx}
\usepackage{bm}
\newtheorem{theorem}{Theorem}[section]
\newtheorem*{theorem*}{Theorem}
\newtheorem{lemma}[theorem]{Lemma}
\theoremstyle{theorem}
\newtheorem{proposition}[theorem]{Proposition}
\newtheorem*{proposition*}{Proposition}
\newtheorem{corollary}[theorem]{Corollary}
\newtheorem*{corollary*}{Corollary}

\theoremstyle{definition}
\newtheorem{definition}[theorem]{Definition}
\newtheorem{remark}[theorem]{Remark}

\newtheorem{example}[theorem]{Example}

\def\C{\mathbb{C}}
\def\R{\mathbb{R}}
\def\N{\mathbb{N}}
\def\Z{\mathbb{Z}}

\newcommand{\exend}{\hfill $\Diamond$}

\def\ab{{\mathcal{A}}}
\def\alet{\alpha}
\def\blet{\beta}
\def\clet{\gamma}

\def\word{\pmb{\omega}}
\newcommand{\id}[1]{{\mathsf{#1}}}

\def\seqsp{{\Sigma}}
\def\subs{\mathcal S}
\def\subsdual{{\mathcal S_*}}
\def\domain{\mathsf{A}}
\DeclareMathOperator{\shift}{\sigma}

\def\T{\pmb{\tau}}

\def\edges{\mathsf{E}}
\def\enn{\mathbf{n}}
\newcommand{\plmax}[2]{\overline{\poslist{{#1}}}{(#2)}}
\newcommand{\addr}[1]{\mathbf{A}_{#1}}

\newcommand{\poslist}[1]{{\pmb{\id{a}}}_{{#1}}}
\newcommand{\plmin}[2]{\underline{\poslist{{#1}}}{(#2)}}
\newcommand{\loc}[1]{{\mathsf{T}({#1})}}
\newcommand{\locn}[1]{{\mathsf{T}_n({#1})}}
\def\lft{{\phi}}
\def\Lft{{\Phi}}

\def\bigo{\mathcal{O}}

\def\II{{\mathbf{I}}}
\def\lends{\mathfrak{LEs}}

\def\Partition{\mathcal{B}}
\def\PartitionI{\mathcal{I}}

\def\mtcong{\Phi}
\def\mtconj{\mtcong}
\def\iiet{\mathfrak{F}}
\def\numa{|\ab|}

\def\vvv{{\mathcal{V}}}

\def\fib{\subs_{fib}}

\def\PD{\subs_{PD}}

\def\specmeas{z}

\def\coinj{j_c}
\def\ed{{E_\delta}}
\def\dd{{D_\delta}}
\def\efull{{E}}

\begin{document}

 \title[Flow view and IIET]{Flow views and infinite interval exchange transformations for recognizable substitutions}
 
 \author{Natalie Priebe Frank}
 \address{Department of Mathematics and Statistics, Vassar College, Box 248, Poughkeepsie, NY  12604, USA}
 \email{nafrank@vassar.edu}

\everymath{\displaystyle}

\keywords{Iterated morphisms, automatic sequences, substitution tilings}
\subjclass[2020]{37B10, 37B52}
\maketitle

\centerline{\em For Uwe, whose good humor and serious dedication are to be emulated.}

\begin{abstract}
A flow view is the graph of a measurable conjugacy $\mtcong$ between a substitution or S-adic subshift $(\seqsp,\shift, \mu)$ and an  exchange of infinitely many intervals in  $([0,1], \iiet, m)$, where $m$ is Lebesgue measure.
A natural refining sequence of partitions of $\seqsp$ is transferred to $([0,1],m)$ using a canonical addressing scheme, a fixed dual substitution $\subsdual$, and a shift-invariant probability measure $\mu$.
On the flow view, $\T \in \seqsp$ is shown horizontally at a height of $\mtcong(\T)$ using colored unit intervals to represent the letters. 

The infinite interval exchange transformation $\iiet$ is well approximated by exchanges of finitely many intervals, making numeric and graphic methods possible. We prove that in certain cases a choice of dual substitution guarantees that $\mtcong$ is self-similar.  We discuss why the spectral type of $\mtcong \in L^2(\seqsp, \mu),$ is of particular interest. As an example of utility, some spectral results for constant-length substitutions are included.

\end{abstract}

\section{Introduction}
In this paper we provide a canonical method for constructing measurable conjugacies between a given recognizable substitution or $S$-adic subshift  $(\seqsp,\shift, \mu)$ and an  exchange of infinitely many intervals in  $([0,1], \iiet, m)$, where $m$ is Lebesgue measure. (We say {\em infinite interval exchange transformation} or IIET for short). The flow view is made by visualizing the bi-infinite sequence $\T \in \seqsp$ as a tiling with colored interval tiles shown in $\R^2$ at a height of $\mtcong(\seqsp)$ to create a sort of `stack' of tilings. The sequences are stacked up so that they are close to others that are the same near the origin, which is why all flow views look the most `organized' near the $y$-axis.

The current paper is situated within the field of Aperiodic Order,  to which Uwe Grimm gave the major works \cite{BaakeGrimmAperiodicOrder, BaakeGrimmAperiodicOrderVol2} and countless other contributions. 
Two things that characterize Uwe's work are the thorough analysis of examples and extensive use of visualization to understand physical and mathematical phenomena. This paper carries on some of those traditions.
Not only does the flow view allow us to ``see" the full sequence spaces, the graphs of the IIETs produce a visual representation of the shift action. We show how our scheme is applied to the iconic examples such as the Fibonacci, Thue-Morse, period-doubling, and Rudin-Shapiro substitutions along with sequences that arise from procedures like Chacon's cutting-and-stacking \cite{ChaconWM}. 

Interval exchange transformations have long been of interest (see, for example, the early work \cite{KeaneIET1975} on minimality and ergodic measures) and are studied in relation to translation surfaces. Importantly, it is known \cite{EverythingCutStack1985} that all probability measure preserving transformations are measurably conjugate to cut-and-stack transformations on $[0,1]$  , implying that there is always an IIET representing a symbolic dynamical system in one dimension. 

IIETs (and translation surfaces of infinite genus such as the Chamanara surface \cite{ChamanaraSurface}), though difficult to understand, are of increasing interest.
Minimal shifts with zero topological entropy have been shown to be topologically conjugate to IIETs \cite{lopez_narbel_2017} using a different technique than that presented here.  \cite{Hooper2015} provides an important study of the invariant measures possible for IIETs. A class of flat surfaces from infinite interval exchange transformations is constructed and analyzed in \cite{LindseyTrevino}. Rational IIETs are studied in \cite{HooperRationalIIET}.  IIETs using a lexicographical address scheme have been constructed and analyzed for some examples of substitution sequences (see, e.g., \cite{LopezNarbel2023}).

Here is an outline of the paper along with some of the main results. In section \ref{Sec:symdefs} we set up the basic definitions and notation for a substitution $\subs$ to have a primitive, recognizable, and/or minimal subshift. We also provide the details of the address scheme that underlies our construction. 

In section \ref{sec:mtcong} we construct a map $\mtconj$ we call a ``canonical isomorphism" of the subshift $(\seqsp, \mu)$, where $\mu$ is a shift-invariant probability measure, to $([0,1], m)$, where $m$ is Lebesgue measure.
We do this using the addressing scheme and some fixed `dual' substitution $\subsdual$ (i.e., one whose substitution matrix is the transpose of that of $\subs$). 

\begin{proposition*}[see Proposition \ref{Pf:mtcong}]
Let $\subs$ be a recognizable substitution whose subshift $(\seqsp, \mu)$ is minimal and let $\mtcong: (\seqsp, \mu) \to ([0, 1], m)$ be a canonical isomorphism. Then $\mtcong$ is measure preserving, well-defined and uniformly continuous everywhere,  bijective almost everywhere, and at most $2|\ab|$:1.
\end{proposition*}

In section \ref{sec:IIET} we define the infinite interval exchange transformation $\iiet$ and prove the following.

\begin{theorem*}[\ref{thm:iiet}]
Let $\subs$ be a recognizable substitution with minimal subshift and let $\mtconj: (\seqsp, \mu) \to ([0,1], m)$ be a canonical isomorphism given by $\subsdual$. The IIET $\iiet$ given by $\subsdual$ is defined almost everywhere, and $\mtconj$ is a measurable conjugacy between $(\seqsp,\shift, \mu)$ and $([0,1],\iiet,m)$.
\end{theorem*}

We are also able to provide accurate approximations for the IIET with finite interval exchange transformations. Here $\lambda$ is the Perron-Frobenius eigenvalue of the substitution matrix of $\subs$, $|\domain|$ is the size of the set of addresses in our address scheme, and $|\ab|$ is the size of the original alphabet the substitution.

\begin{corollary*}[\ref{cor:iietest}]
For any $n \in \N$ there is an exchange of $n(|\domain| - |\ab|) + |\ab|$ 
intervals that is equal to $\iiet$ on all but $|\ab|$ intervals of total measure $\le \lambda^{-n}$.
\end{corollary*}

The graphs of canonical IIETs have a tendency to look like they have some form of self-similarity. When $\subs$ satisfies a simple combinatorial condition called `proper', there is a choice of canonical isomorphism for which the corresponding IIET is formally self-similar in the following sense.

\begin{proposition*}[\ref{prop:self-similarilty}]
Let $\subs$ be a recognizable and proper substitution with minimal subshift, and suppose $|\subs(\alet)|>1$ for each $\alet \in \ab$. Then $\subs$ has a canonical IIET $\iiet$ for which there is a constant $\kappa \in [0,1)$ such that \[
\iiet(x) = \lambda(\iiet(x/\lambda) + \kappa) \text{ for a.e. } x \in [0,1].
\]
\end{proposition*}

In section \ref{sec:spectral} we discuss the interaction of our construction with spectral theory. We include the following two results for constant-length substitutions.

\begin{proposition*}[\ref{prop:constantlengthfinite}]
If $\subs$ is a primitive recognizable constant-length substitution with expansion factor $K$ then $\left(\iiet^{K^n}(x)\right)_{n=1}^\infty$ has at most $|\ab|$ accumulation points for a.e. $x \in [0,1]$. 
\end{proposition*}

\begin{theorem*}[\ref{thm:coincthm}] Let $\subs$ be a primitive recognizable substitution of constant length $K>1$ with IIET $\iiet$. Then  $\lim_{n \to \infty}\iiet^{K^n}(x) =x$ almost everywhere if and only if $\subs$ has a coincidence.
\end{theorem*}

We believe there may be analogues to these results in the non-constant length case. Several graphs are provided that illustrate how powers of the IIETs behave for systems with different dynamical/spectral properties.

In section \ref{Sec:generalizes-adic} we describe how to adapt the construction to S-adic systems. We also show how a flow view for a substitution can be adapted to any tiling space made from interval tiles representing the letters. 
Again we provide examples that illustrate how the dynamics of a system are represented in the flow view.  We conclude the paper in section \ref{sec:questions} with a selection of questions and observations raised by this investigation.

Our construction has several features: 

\smallskip
\noindent{\bf Efficiency.} For each $\epsilon > 0$ a subset of Lebesgue measure $\ge 1-\epsilon$ is contained in $\bigo(|\ln(\epsilon)|)$ intervals in the domain of $\iiet$. We will show that for any $n \in \N$,  there is an {\bf interval exchange tranformation $\iiet_n$} that is equal to $\iiet$ on $\bigo(n)$ intervals and differs on a set of measure $\le \lambda^{-n}$ for some $\lambda>1$. Figure \ref{IIETIntro} shows  $\iiet_7$ for the {\bf Fibonacci} ($\alet \mapsto \alet\blet, \, \blet \mapsto \alet$) and $\iiet_5$ for {\bf Thue--Morse} ($\alet \mapsto \alet\blet,  \,\blet \mapsto \blet\alet, $) substitutions. These iterations are high enough that the graphs shown don't differ much from the IIET to the naked eye.
\exend
\begin{figure}[ht]
\centering
\includegraphics[width=.35\textwidth]{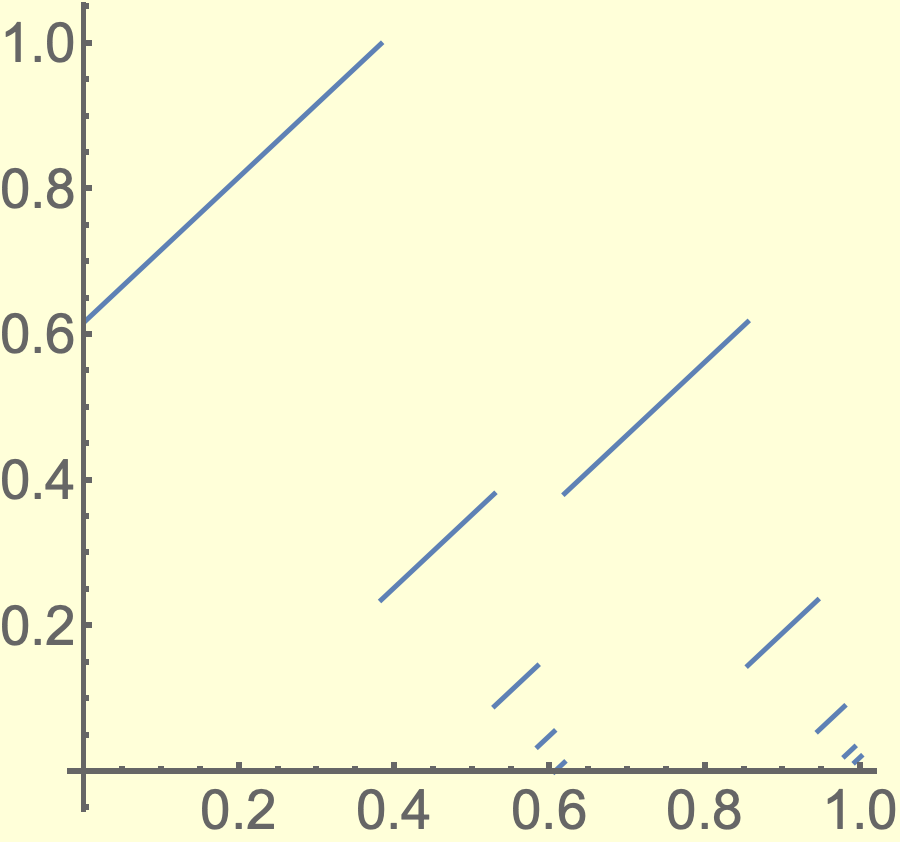} \hskip 1em \includegraphics[width=.35\textwidth]{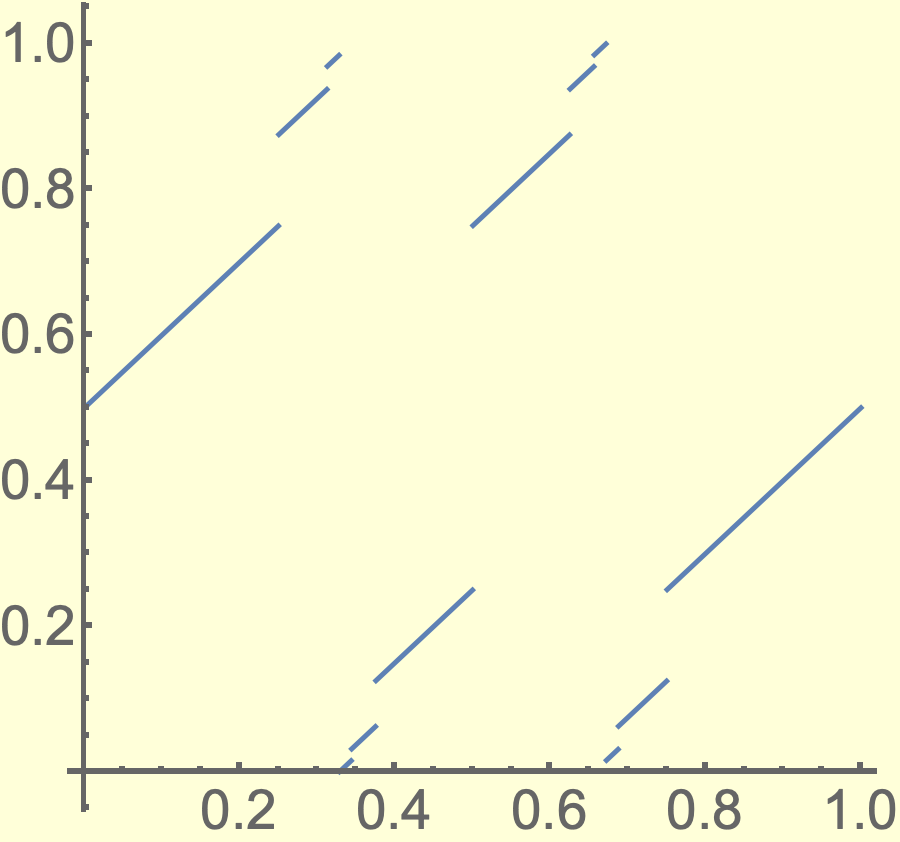}
\caption{(Approximations to the) canonical IIETs for Fibonacci (left) and Thue--Morse (right) substitutions.}
\label{IIETIntro}
\end{figure}

\smallskip
\noindent{\bf Visualization.} The graph of the map $\mtcong: \seqsp \to [0,1]$, called a {\bf \em flow view}, is a picture of every sequence in $\seqsp$ stacked up between $0$ and $1$ on the $y$-axis. It literally graphs the a.e.\ one-to-one correspondence between $[0,1]$ and the subshift by showing each $\T \in \seqsp$  in colored unit interval tiles at a height of $\mtcong(\T)$. 
Figure \ref{FlowViewIntro} shows the central portion of the flow views for the Fibonacci and  Thue--Morse substitutions. The vertical black line is the interval from 0 to 1 on the $y$-axis. In each, the line at $1/e$ is shown. The sequence/tiling represented by $1/e$ is the intersection of this line with the flow view, with $a$ as green and $b$ as blue unit intervals, where $1/e$ was chosen arbitrarily because it is irrational with respect to the expansion constants. The flow view indicates that for the Fibonacci, this $\T= \ldots \beta \alpha \alpha \beta \alpha .\alpha \beta \alpha \beta \alpha \alpha   \ldots$. The Thue-Morse flow view indicates that the $\T$ going to $1/e$ must be $\T = \ldots \beta . \alpha \alpha \beta \beta \alpha \ldots$.
\exend

\begin{figure}[ht]
\centering
\includegraphics[width=.9\textwidth]{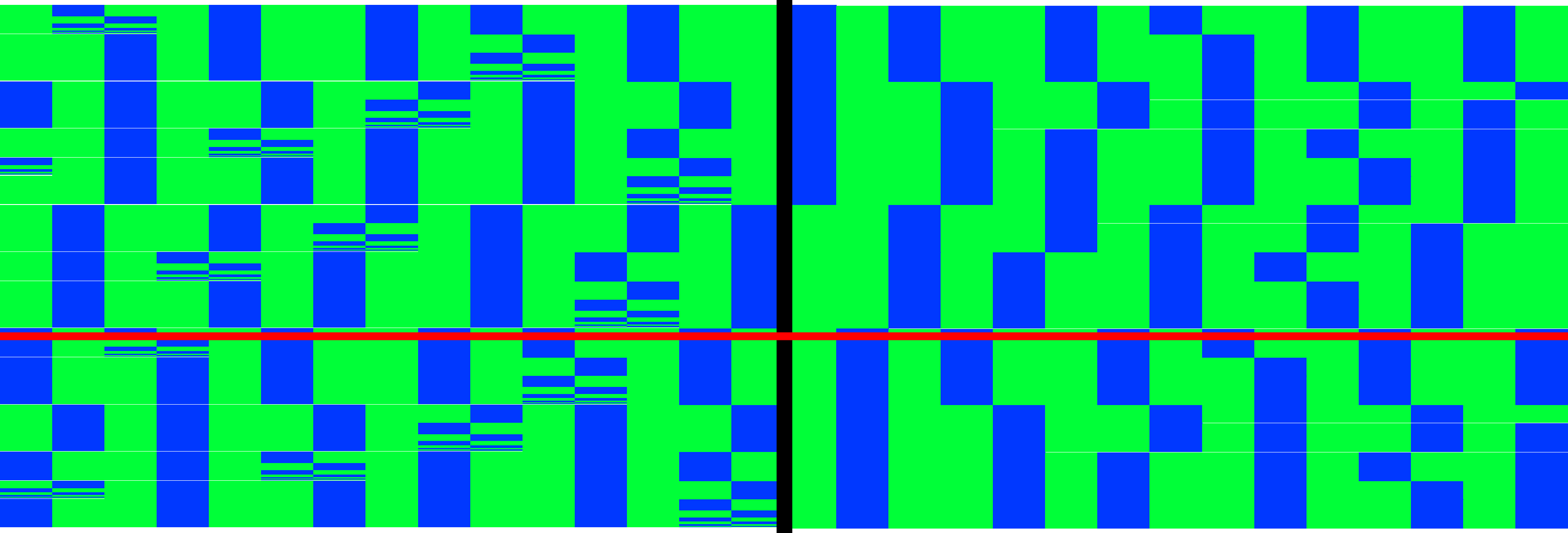}\\\vskip 1em
\includegraphics[width=.9\textwidth]{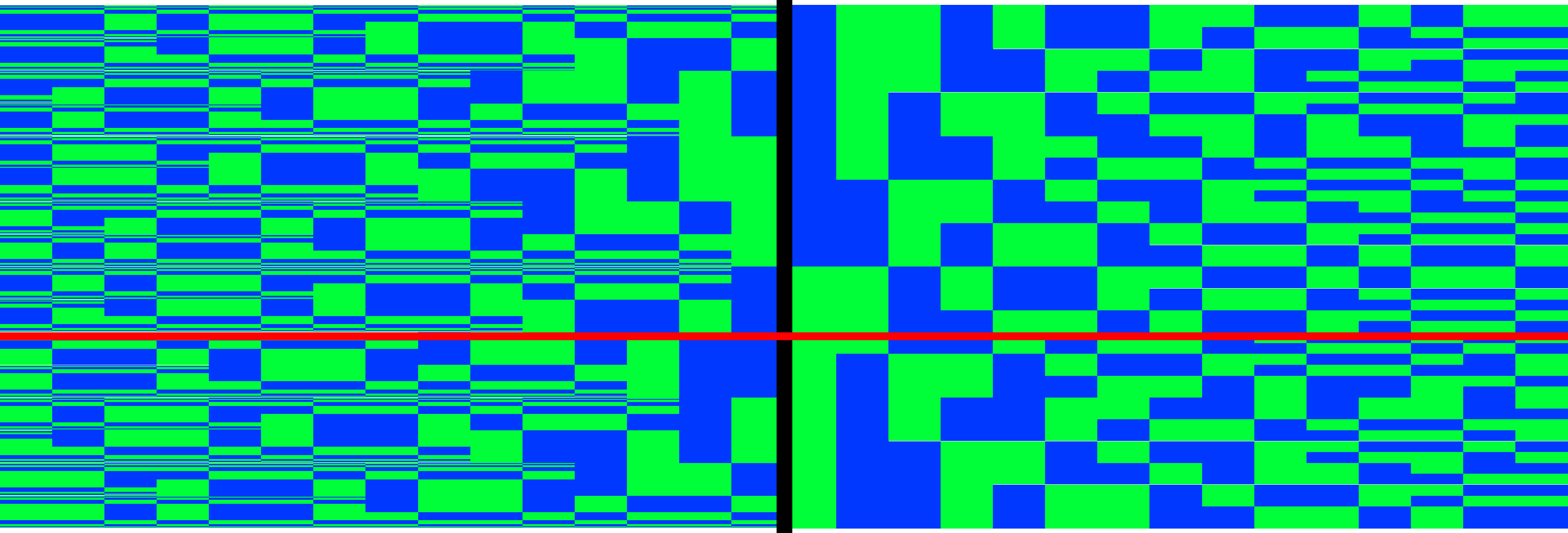}
\caption{Flow views for Fibonacci (top) and Thue--Morse (bottom) subshifts. The red line highlights the $\T\in \seqsp$ for which $\mtcong(\T) ={1}/{e}$.}
\label{FlowViewIntro}
\end{figure}
\smallskip
\noindent{\bf Generalizations.} The construction works for a large class of {\bf S-adic systems} and the adaptations necessary are provided. The continuous analogues, self-similar and fusion tilings of $\R$, are suspensions and therefore the results apply for a transversal. The construction works in some higher dimensional situations to produce commuting IIETs on $[0,1]$. \exend

\smallskip
\noindent{\bf Applications.} The {\bf spectral theory} of the subshift can be investigated in $L^2([0,1])$ directly. Moreover, $\mtcong$  is a natural element of $L^2(\seqsp, \mu)$ whose spectral decomposition is particularly important.  A few results are presented for constant-length substitutions to show the possibilities. Because of the intimate connection between {\bf translation surfaces} and interval exchange transformations, our IIETs provide an unlimited stable of translation surfaces that are probably of infinite genus.
(We do not know a simple proof of infinite genus but it is likely to be the case for most substitutions.) The graphs of our IIETs show types of {\bf self-similarity properties}, which is expected to simplify the basic analysis of their translation surfaces.
\exend

\section{Setting: symbolic and substitution dynamical systems}
\label{Sec:symdefs}
We briefly review and set notation for substitution sequences and introduce some necessary terminology. For a more thorough introduction see \cite{Kitchens, Fogg} and the recent survey \cite{akiyama2020substitution}. 
 
\subsection{Symbolic dynamics}\label{sec:symbdef}

A finite set $\ab$ is taken to be the {\bf \em alphabet} with elements $\alet \in \ab$ known as {\bf \em letters}. A {\bf \em word} is a function $\word: \{j, j+1, \ldots, k\}\to \ab$ where $j \le k \in \Z$. The {\bf \em length} of a word is $|\word|=k-j+1$. We use the notation $\word[i, ..., n]$ for a {\bf \em subword} of $\word$, where $j \le i \le n \le k$ is assumed. We use ordinary function notation $\word(i)$ and refer to the element of $\word$ at $i$. A sequence in $\ab^\Z$ is an infinite word.

The set $\ab^+$ is the set of all words of finite length and $\ab^\Z$ is the set of all (bi)infinite sequences on $\ab$. We endow $\ab^\Z$ with a  metric that determines the distance between $\T, \T' \in \ab^\Z$ based on the largest interval around the origin on which they are identical. Any choice will generate the product topology on $\ab^\Z$. We use the following. 

\begin{definition}
For any  $\T$ and $\T' \in \ab^{\Z}$ with $\T \neq \T'$, define $N(\T, \T')=\inf \{n \geq 0 : \T[-n, \ldots, n] \neq \T'[-n, \ldots, n]\}.$
 Then 
 \[d(\T,\T')=\begin{cases}
  \exp(-N(\T,\T'))& \T\neq\T' \\
  0 &\T=\T' \end{cases}. \] 
  \end{definition}

There is a {\bf \em shift} action $\shift: \ab^\Z \to \ab^\Z$ by elements of $\Z$ given by
\begin{equation}
\left(\shift(\word)\right)(k) = \word(k + 1) \quad \text{ for all } k \in \Z.
\end{equation}
The shift moves $\word$ one unit to the left, so that whatever was at $1$ in $\word$ is at $0$ in $\shift(\word)$. The shift is continuous and the resulting dynamical system is known as the {\bf \em full shift} $(\ab^\Z,\shift)$. Any closed, shift-invariant subset $\seqsp \subset \ab^\Z$ inherits the metric topology and shift action; the pair $(\seqsp,\shift)$ is known as a {\bf \em subshift}.

\subsection{Substitutions and their dynamical systems (see \cite{Fogg,Queffelec})}
 A {\bf \em  substitution rule} for $\ab$ is a rule that takes each letter in $\ab$ to a nonempty finite word\footnote{These are also called non-erasing morphisms.}. For each $\alet \in \ab$ we write $\subs(\alet) = \alet_0 \alet_2 ... \alet_{\ell-1}$, where $\ell = |\subs(\alet)|\ge 1$, and the domain of $\subs(\alet)$ is $\{0, 1, \ldots, \ell-1\}$ . 
 
We think of applying a substitution to the word $\word$ by concatenation, but we need to be a little careful about the domain.  Borrowing from tiling terminology, we will call $\subs^n(\alet)$ an {\bf \em $n$-supertile of type} $\alet$ and define it recursively by:
\begin{equation}
\subs^n(\alet) = \subs^{n-1}(\alet_0) \, \subs^{n-1}(\alet_1)\, ... \,\subs^{n-1}(\alet_{\ell-1}), 
\label{eqn:supertiledef}
\end{equation}
where the domain of $\subs^n(\alet)$ is $\left\{0, 1, \ldots, \sum_{i = 0}^{\ell - 1} | \subs^{n-1}(\alet_i)|-1\right\}$.
\vskip 1 em
Let $\mathcal{L}$ be the set of all subwords of all supertiles of $S$. Then $\mathcal{L}$ determines a subshift of $\seqsp \subset \ab^\Z$ by proclaiming
$\T \in \ab^\Z$ to be {\bf \em admitted} by $\subs$ if and only if every finite subword of $\T$ is, up to a shift of its domain, an element of $\mathcal{L}$. 

\begin{definition}
When the set $\seqsp = \{\T \in \ab^\Z \, | \, \T \text{ is admitted by } \mathcal{L}\}$ is nonempty, endowed with the subspace topology and shift $\shift$, and a shift-invariant Borel probability $\mu$, we call the triple
 $(\seqsp,\shift,\mu)$ the {\bf \em subshift of} $\subs$ or more generally a {\bf \em substitution subshift}.
\label{def:substitutionsubshift}
\end{definition}

We would like to define an action of substitution on elements of $\seqsp$. If we use sequence notation to write $\T=(\alet_n)_{n \in \Z}$, we certainly can define $\subs(\T)$ to be $\ldots \subs(\alet_{-1}) \subs(\alet_0) \subs(\alet_1)\ldots$, but there is not a natural location for the start of $\subs(\alet_0)$ (or any other supertile). For simplicity we define it so that the domain of the supertile $\subs(\alet_0)$ within $\subs(\T)$ begins at $0$. This action is not invertible, since there are elements in $\seqsp$ where the origin is not at the beginning of a substituted tile. For any $n \in \N$, every element of $\seqsp$ takes the form $\shift^k(\subs^n(\T))$ for some $\T \in \seqsp$, $\alet_0 \in \ab$, and $k \in 0, 1, \ldots, |\subs^n(\alet_0)| -1$ (\cite[Prop 6.2]{CanteriniSiegel2001prefix}).  The notion of recognizability tells us when there is a type of `desubstitution' on sequences: 

\begin{definition}
We say a substitution is {\em recognizable} if for every $\T \in \seqsp$ there is a unique $\T'$  and some $0 \le k \le |\subs(\T'(0))|-1$ for which $\T = \shift^k(\subs(\T'))$.
\end{definition}
Recognizability extends to supertiles of any level: for any $\T \in \seqsp$ there is a unique $\T'$ and $k \in \{0,1, \ldots |\subs^n(\T')| -1\}$ with $\T = \shift^k(\subs^n(\T'))$.

The {\bf \em cylinder set} for $\alet \in \ab$ is defined to be 
\begin{equation}
[\alet] = \left\{\T \in \seqsp \text{ with } \T(0) = \alet\right\}.
\label{eqn:aletcylinder}
\end{equation}
We define $\subs^n([\alet]) = \{\subs^n(\T), \T \in [\alet]\}$. The topology of $\seqsp$ is generated by sets of the form $\shift^k(\subs^n([\alet])),$ where $\alet \in \ab$ and $0 \le k \le |\subs^n(\alet)| -1$  (see e.g. \cite{DurandHostSkau1999, CanteriniSiegel2001prefix}). Sets of the form $\shift^k(\subs^n([\alet]))$ with  $0\le k \le |\subs^n(\alet)|-1 $ are called {\bf \em $n$-cylinders}.

\begin{example}
Our main illustrating example will be the {\em period-doubling} substitution $\PD$ defined by $A \mapsto AB, B \mapsto AA$. We compute the first few supertiles of type $A$ to be
\[A \mapsto AB \mapsto ABAA \mapsto ABAAABAB \mapsto \cdots.\]
Sequences in $[A]$ take the form $\T=\cdots .A\cdots,$ where the decimal point is intended to denote the location of the entry at 0. Sequences in
the $3$-cylinder $\shift^3(\PD^3([A])$ take the form \[\T = \cdots ABA.AABAB\cdots.\]
The period-doubling substitution is clearly recognizable: wherever a $B$ appears in $\T$, it must be the second letter of an $A$ supertile. Knowing the location of just one $B$ in $\T$ thus determines where all the 1-supertiles of $\T$ are.
\end{example}

\begin{remark}
Depending on the substitution, there can be a difference between the set of all sequences for which the word beginning at the 0 is equal to $\subs(\alet)$ and $\subs([\alet])$ as we defined it. This can happen if there are allowable words that are the same word as $\subs(\alet)$ but that are not actually supertiles by recognition. For instance, the word  $AA$ appears in elements of the period-doubling substitution subshift both as $\subs(B)$ but also in the middle of $\subs(B)\subs(A)$.
Our use of the term `cylinder' in `$n$-cylinder' may thus be nonstandard.
\end{remark}

Suppose we write $\ab = \{\alet_1, \alet_2, \ldots, \alet_{|\ab|}\}$.
The {\bf \em transition\footnote{also known as the {\em substitution matrix} or {\em abelianization} of the substitution} matrix $M$}  of $\subs$ is the $|\ab| \times |\ab|$ matrix with entries $M_{ij}$ given by the number of times $\alet_i$ appears in the word $\subs(\alet_j)$. It is clear that $M$ is nonnegative; $M$ is said to be {\bf \em primitive} if there is some $k$ for which every entry of $M^k$ is positive. Primitive or not, there is a positive eigenvalue $\lambda$ called the {\bf \em Perron-Frobenius eigenvalue} for which $|\lambda| \ge |\lambda'|$ for any other eigenvalue $\lambda'$ of $M$. We call $\lambda$ the {\bf \em expansion factor} of $\subs$.

The matrix $M$ keeps track of the lengths of supertiles: \[(1, 1, \ldots,1) M^n = (|S^n(\alet_1)|, |S^n(\alet_2)|, \ldots, |S^n(\alet_{\numa})|).\]
A right eigenvector for $\lambda$ represents the relative frequencies of letters in $\ab$ for an ergodic component of $\seqsp$.
In the S-adic case, frequencies will be compatible with the sequence of transition matrices in a similar way (see \cite{akiyama2020substitution, BezuglyiKarpelSurvey}). Given a nonnegative right probability eigenvector $\vec{r}$ for $\lambda$, there is an invariant measure $\mu$ such that 
\begin{equation}
\mu([\alet_j]) = \vec{r}(j) \text{ and } \mu(\subs^n([\alet_j])) = \vec{r}(j)/\lambda^n\text{ for all } j = 1, 2, ..., \numa. \label{eq:evectmeasure}
\end{equation}  When $M$ is primitive, $\lambda$ is unique and has a unique right probability eigenvector with no zero entries. 

\begin{example}
Continuing with the period-doubling substitution $\PD$ defined by $A \mapsto AB, B \mapsto AA$, we compute
the transition matrix to be $M = \left(\begin{smallmatrix}
1 & 2 \\ 1 & 0
\end{smallmatrix}\right)$ with right Perron eigenvector $(2/3, 1/3)$. The subshift has a unique shift-invariant probability measure $\mu$ for which $\mu([A]) = 2/3$ and $\mu([B]) = 1/3$. The measures of other Borel sets can be deduced from the fact that $\mu$ is shift-invariant and that for $C \in \{A, B\}$, $\mu(\PD^n([C])) = 2^{-n}\mu([C])$.
\exend
\end{example}

\subsection{The address scheme}
\label{sec:addreses}
Here we develop the address scheme needed for our construction. It is fairly standard, appearing for instance in \cite{CanteriniSiegel2001prefix}, but we give it notation adapted to our purposes. For each $n = 0, 1, 2, ...$ we know that if $\subs$ is recognizable then every $\T \in \seqsp$ is an element of $\shift^k(\subs^n([\alet])$ for a unique $\alet \in \ab$ and $k \in \{0, 1, \ldots, |\subs^n(\alet)|-1\}$. This means that the set of all $n$-cylinders forms a partition of $\seqsp$ which we denote
\begin{equation}
\Partition_n = \left\{\shift^k(\subs^n([\alet])), \alet \in \ab \text{ and } 0\le k \le |\subs^n(\alet)|-1 \right\}.
\label{eqn:supertilepartition}
\end{equation}

The fact that the partition $\Partition_{n+1}$ is a refinement of the partition of $\Partition_n$ can be seen as follows. Any two elements of the $(n+1)$-cylinder
$\shift^k(\subs^{n+1}([\alet])$ have the origin in precisely the same location of the same $(n+1)$-supertile. That immediately means that they must also share the same $n$-supertile at the origin. Thus any given $n$-cylinder either completely contains or is completely disjoint from every $(n+1)$-cylinder.

\begin{definition}\label{def:n-cylinder}
Let $\seqsp$ be given by a recognizable substitution. The refining sequence of partitions $\left(\Partition_n\right)_{n=0}^{\infty}$ is called the {\bf \em canonical partition sequence} of $\seqsp$. 
\end{definition}

The set of {\bf \em address labels} is the subset of $\ab \times \N$ given by 
\begin{equation}\label{def:domain}
\domain = \{(\alet,j) \text{ such that } 0 \le j \le |\subs(\alet)|-1\}.
\end{equation} We think about the {\bf \em type} of a supertile addressed by $(\alet, j)$ as being $\alet$ and the {\bf \em location} in that supertile as $j$.
Elements of $\domain$ are used in two crucial ways. One is to specify the $1$-cylinder $[\subs(\alet,j)] := \shift^j(\subs([\alet]))$. The other is to identify the the letter {\bf \em in the $j$th position of} $\subs(\alet)$, which is the letter at $0$ in all elements of $\shift^j(\subs([\alet]))$ and is given by $\subs(\alet)(j)$.

The position of the supertile at the origin in any $\T\in \seqsp$ is indexed by the domain $\domain$ and is, by recognizability, 
unique. The unique $\alet \in \ab$ and $j \in \{0, 1, \ldots, |\subs(\alet)|\}$ such that $\T \in \sigma^j(\subs([\alet]))$ is defined to be the $1${\bf \em -address of $\T$}  and we write $\poslist{1}(\T) =(\alet, j) \in \domain$. In this case, in somewhat of an abuse of notation/terminology, we also say that $\T(0)$ is
is in the $j$th spot of the $1$-supertile $\subs(\alet)$.  The $1${\bf \em -cylinder addressed by $\id{a} =(\alet,j)$} is defined to be the canonical partition element
\[
[\subs(\alet,j)]=\{\T \in \seqsp \text{ such that } \poslist{1}(\T) = (\alet,j)\}=
\shift^{j}(\subs([\alet])).
\]

By recognizability, the 1-supertile at the origin in $\T$ sits inside of a unique 2-supertile in a way that is also labeled by $\domain$. 
Thus for any $\T \in \seqsp$ we can define the {\bf \em 2-address of $\T$} to be $\poslist{2}(\T) = ((\alet_1, j_1),(\alet_2, j_2))$ if the 1-supertile at the origin in $\T$ is of type $\alet_1$ and is in position $j_1$, and that supertile is contained in a $2$-supertile of type $\alet_2$ at position $j_2$. We give this supertile the notation $\subs((\alet_1, j_1),(\alet_2, j_2)) $ and define the {\bf \em $2$-cylinder addressed by $((\alet_1, j_1),(\alet_2, j_2))$} to be
\[[\subs((\alet_1, j_1),(\alet_2, j_2))] = \shift^k(\subs^2([\alet_2])) \text{ where } k = \sum_{i = 0}^{j_2-1}|\subs\left(\subs(\alet_2)(i)\right)| + j_1.\] 

In general, recognizability implies that for all $n \in \N$ the position of the $(n-1)$-supertile of $\T$ with 0 in its domain must be contained in a unique $n$-supertile and this can be labelled by $\domain$.
Every $\T \in \seqsp$ contains a nested sequence of $n$-supertiles containing the origin that tells us which canonical partition elements it belongs inside.  For any given $n$, we record this as the {\bf \em $n$-address of $\T$}, denoted $\poslist{n}(\T) =((\alet_1, j_1), \ldots, (\alet_n, j_n)) \in \domain^n$.   When $\poslist{} = ((\alet_1, j_1), \ldots, (\alet_n, j_n))$, the length of $\poslist{}$ is the level of the supertile $\poslist{}$ represents and the {\bf \em type of $\poslist{}$} is $\alet_{n}$; $\poslist{}$ specifies an exact location (or domain) for $\subs^n(\alet_n)$.
The address also corresponds to an element of the $n$th partition in the canonical partition sequence that we call the {\bf \em $n$-cylinder addressed by $((\alet_1, j_1), \ldots, (\alet_n, j_n)),$} given by \[ [\subs((\alet_1, j_1), \ldots, (\alet_n, j_n))]= \shift^k(\subs^n([\alet_n])), \] where $k$ is computed as a sum of lower-level supertile lengths as it was for 2-cylinder case. Such sums appear in \cite{CanteriniSiegel2001prefix}.

For any $\T \in \seqsp$, $\T(0)$ is determined by $\poslist{1}(\T)$; that means the 0-cylinder of type $\alet$ is the union of 1-cylinders that have $\alet$ as the letter at 0. 
For each $\alet \in \ab$, the set of all positions $\alet$  appears in 1-supertiles is 
\begin{equation}
\loc{\alet}=\{(\blet,j) \in \domain \, | \, \subs(\blet)(j) = \alet\}
=\{(\blet, j) \in \domain \, | \, \alet \text{ is the } j\text{th letter of }\subs(\blet ) \}.
\label{eqn:parentset}
\end{equation}

We can write the 0-cylinder $[\alet]$ as the union of all 1-cylinders that have an $\alet$ at 0:\[[\alet] = \bigcup_{(\blet,j) \in \loc{\alet}} [\subs(\blet, j)].\]

\begin{definition}
We say $\poslist{} = ((\alet_1,j_1), (\alet_2, j_2),...)$ is an {\bf  \em address} if $(\alet_k, j_k) \in \loc{\alet_{k-1}}$ for all $1 < k \le |{\poslist{}}|,$ where we allow $|\poslist{}| = \infty$.  The set of all addresses of lengths $n$, $\infty$, or ``any'' are denoted $\addr{n},\addr\infty,$ and $\addr{}$, respectively.
\end{definition}

An address can be thought of as instructions: First, place a supertile of type $\alet_1$ so that the origin is in the $j_1$th spot. Then, slide a copy of $\subs^2(\alet_2)$ to match its $j_2$'th 1-supertile to the one in place already. Then move a copy of $\subs^3(\alet_3)$ to match its $j_3$th 2-supertile to the existing one, and so on. 

\begin{example}
For the period-doubling substitution $A \to AB, B \to AA$ we have the address set with simplified notation \[\domain = \{(A, 0), (A, 1), (B, 0), (B, 1)\} = \{A0, A1, B0, B1\}.\] Consider $\T$ such that $\poslist{3}(\T)=\{B1, A1, A0\}$. Then we know that the 1-supertile at the origin is labeled by $B1=(B,1)$ and so is of the form $\PD(B) = AA$. Thus $\T$ looks like $ \cdots A.A \cdots$ and $\T\in \shift^1(\PD([B]))$. The 2-supertile at the origin is labeled by $A1 = (A, 1)$ and so is of the form $\PD^2(A) = ABAA$ where  the supertile at 0 is in the right half of its 2-supertile. It must be that  $\T = \cdots ABA.A \cdots \in \shift^3(\PD^2([A]))$. Finally, the address $A0 = (A, 0)$ tells us that the 2-supertile is in a 3-supertile of type $A$ in the left half. We know $\PD^3(A) = ABAAABAB$ so we see that  $ \T = \cdots ABA.AABAB\cdots \in \shift^3(\PD^3([A]))$.
Figure  \ref{fig:Useaddress} illustrates the process for the 3-supertile $\PD(B1, A1, A0)$ using green unit intervals to represent $A$s and blue unit intervals to represent $B$s.
\begin{figure}[ht]
\includegraphics[width=.45\textwidth]{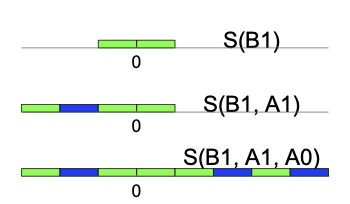}
\caption{Building a supertile from an address string.}
\label{fig:Useaddress}
\end{figure}
\end{example}

Any two addresses  $\poslist{}$ and $\poslist{}'$ that share a common prefix $\poslist{n}$ correspond to elements of $\seqsp$ that have the same $n$-supertile in the same location at the origin. Thus both $[\subs(\poslist{})]$ and $[\subs(\poslist{}')]$ are contained in $[\subs(\poslist{n})]$, the $n$-cylinder of all sequences with the corresponding supertile in the corresponding location. Each $\poslist{} \in \addr {n}$ addresses an $n$-supertile of some type $\alet_{n}$ and therefore can be contained in any 
$(n+1)$-supertile location from $\loc{\alet_{n}}$. That means
\begin{equation}
[\subs(\poslist{})] = \bigcup_{(\blet, j) \in \loc{\alet_n}} [\subs(\poslist{},(\blet,j)) ].
\label{eq:supertiletransition}
\end{equation}
This shows the precise way in which $\Partition_n$ refines $\Partition_{n-1}$ for all $n \in \N$. Since the measures of all partition elements go to zero, the sequence refines to points almost everywhere.  Note that extending an address corresponds to identifying a higher-order supertile around 0.

Addresses may fail to uniquely specify an element of $\seqsp$, but this can only happen when a sequence is the concatenation of two infinite-order supertiles, one extending off to infinity on the right and one to the left. In that case there may be at most $|\ab|$ sequences with the same address, and this is a measure-0 event. More generally we have the following lemma.

\begin{lemma}\label{lem:meas0}
If $(\seqsp, \shift)$ is a minimal subshift given by a recognizable substitution and $\mu$ is shift invariant, the subset 
\begin{equation}
\seqsp_0 = \{\T \in \seqsp \, | \, \text{each }\id{a} \in \domain \text{ appears infinitely often in } \poslist{}(\T)\}
\label{eqn:seqspcenter}
\end{equation}
has full measure.
\end{lemma}

The proof of this folk lemma is relatively technical, so we sketch it here. Note that both $\seqsp_0$ and its complement are shift invariant. To show the complement of $\seqsp_0$ must have measure 0, notice that it is contained in the set $K$ of all $\T \in \seqsp$ for which $\poslist{}(\T)$ has finitely many  $(\alet, 0)$ for a fixed $\alet \in \ab$. Showing that $K$ has measure 0 requires a series of estimates.

\section{Canonical isomorphisms}
\label{sec:mtcong}

In this section we use our address scheme to define a measurable conjugacy $\mtcong: (\seqsp, \mu) \to ([0,1],  m)$, illustrating the process with the period-doubling substitution. Readers may well be reminded of many constructions using similar ideas, for instance \cite{SceneryFlow, DownarowiczMaass2008}, Anosov flows, Veech rectangles, and a variety of tower-related constructions. We also define the level-$n$ flow view and the full flow view.

The {\bf standard set of assumptions} to be used are as follows. The substitution $\subs$ should be recognizable and its subshift should be {\bf \em minimal} in the sense that every orbit is dense. The measure $\mu$ is assumed to be a shift-invariant Borel probability measure. With these assumptions, letters with zero frequency and other technical difficulties are avoided.

Here is an overview of how we will define $\mtcong$. We construct a partition sequence $(\PartitionI_n)$ of $[0,1)$ to match up with the
canonical partition sequence $(\Partition_n)$ of $\seqsp$.  A partition element $\II(\poslist{n})$ in $[0,1)$ of length $\mu([\subs(\poslist{n})])$ is assigned to each $n$-cylinder in a way that preserves inclusion. The partition sequence in $\seqsp$ refines to points so the map $\mtconj$ can be thought of intuitively as infinite intersections, but we define it using limits of left endpoints given by partial sums $\mtconj_n: \addr n\to [0,1)$.
We will define the $n$-th level flow view to display the supertile $\subs(\poslist{n})$ horizontally at $y=\mtcong_n(\poslist{n})$ as colored intervals with vertical thickness $\mu([\subs(\poslist{n})])$.

To start, we need to choose an initial partition of $[0,1)$. Since $\mu$ is a probability measure we know $\sum_{\alet \in \ab} \mu([\alet]) = 1$, so for each $\alet \in \ab$, choose a left endpoint $\Lft_0(\alet)\in [0,1)$ such that the intervals $\II(\alet):=[\Lft_0(\alet), \Lft_0(\alet)+ \mu([\alet]))$ cover $[0,1)$.  The initial partition is $\PartitionI_0=\{\II(\alet), \alet \in \ab\}$. 

Recall from equation \eqref{eq:evectmeasure} that 
for $(\blet, j) \in \domain{}$, $\mu([\subs(\blet,j)]) = \mu(\subs([\blet]))=\mu([\blet])/\lambda,$ where $\lambda$ is the Perron-Frobenius eigenvalue of the substitution matrix. We have
\begin{equation}
\mu([\alet]) = \sum_{(\blet, j) \in \loc{\alet}} \mu([\subs(\blet)])=  \sum_{(\blet, j) \in \loc{\alet}}\mu([\blet])/\lambda.
\label{eqn:transitionconsistent}
\end{equation}
That means for each $\alet \in \ab$ we can partition $[0, \mu([\alet]))$ into intervals of these lengths. For each $\alet \in \ab$ let the `left endpoint' function $\lft: \loc{\alet} \to [0, \mu([\alet]))$ record the left endpoints of this partition. We have
\begin{equation}
\left[0,\mu([\alet])\right) = \bigcup_{(\blet, j)\in \loc{\alet}} \left[\lft(\blet, j), \, \lft(\blet, j) +  \frac{\mu([\blet])}{\lambda}\right).
\label{eqn:le(aj)}
\end{equation}

The orders of the subintervals chosen for $\lft$ are conveniently expressed as a fixed {\bf \em dual substitution} $\subsdual$, defined as a substitution on $\ab$ whose transition matrix is the transpose of that of $\subs$. Note that there may be several dual substitutions to choose from and there does not seem to be a natural choice. 
The row corresponding to $\alet$ in the substitution matrix of $M$ is the column for $\alet$ in the matrix of $\subsdual$, so $\subsdual(\alet)$ contains the composition of letters seen in $\loc{\alet}$. 
To move the partition of $[0,\mu([\alet]))$ into the correct location in $[0,1)$ we add $\Lft_0(\alet)$. Let $(\blet, j) \in \domain$ with $(\blet, j) \in \loc{\alet}$. We define
\[\Lft_1(\blet, j) = \Lft_0(\alet) + \lft(\blet, j) \, \text{ and }  \, \II(\blet, j) = \left[\Lft_1(\blet, j),  \Lft_1(\blet, j) + \mu([\subs(\blet, j)])\right)\]
So $ \PartitionI_1 = \left\{\II(\blet, j) \, | \, (\blet, j) \in \domain\right\}$ refines partition $\PartitionI_0$ and $m(\II(\blet, j)) = \mu([\subs(\blet, j)]) = \mu([\blet])/\lambda$ for all $(\blet, j) \in \domain$.

\begin{example} Consider the period-doubling substitution $A \mapsto AB, B \mapsto AA$. We know $\mu([A]) = 2/3$ and $\mu([B]) = 1/3$, so we define
$\Lft_0(A) = 0, \, \Lft_0(B) = 2/3$. 
Any element $\T \in \seqsp$ with $\T(0) = A$ will ultimately be sent somewhere in the interval $[0, 2/3]$. Likewise if $\T(0) = B$ then $ \mtcong(\T) \in [2/3, 1]$. 

We have \[\loc{A} = \{(A, 0), (B, 0), (B, 1)\} =\{A0, B0, B1\}\text{ and }\loc{B} =\{(A, 1)\}= \{A1\}.\] and our choice for  the dual substitution is $\subsdual: A \mapsto ABB $ and  $B \mapsto A.$ 
Figure \ref{fig:phipartition} shows the choice of $\lft$ that makes.
\begin{figure}[ht]
\includegraphics[width=.35\textwidth]{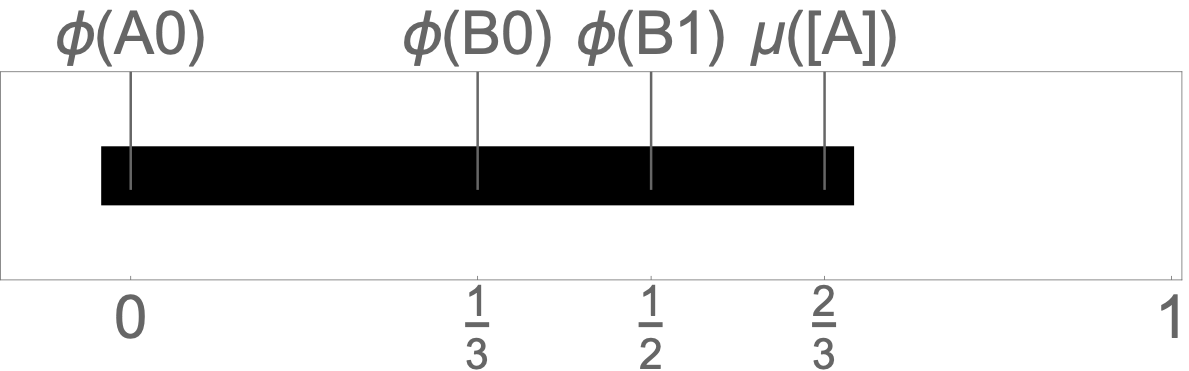} \hskip 1em \includegraphics[width=.35\textwidth]{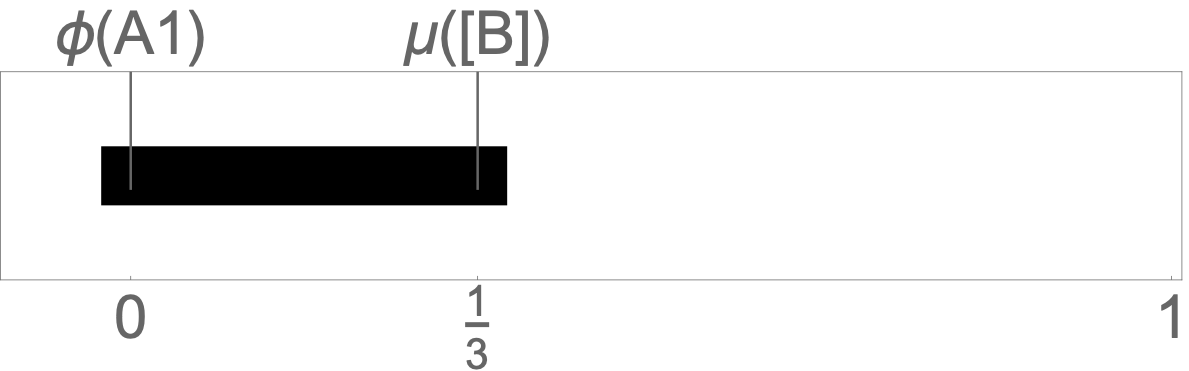}
\caption{Our choice of initial partition for $\PD$.}
\label{fig:phipartition}
\end{figure}

For $\lft$ and $\Lft_1$ that gives:
\[\Lft_1(A0) = 0 + 0, \quad \Lft_1(B0)= 0+ 1/3, \quad \Lft_1(B1) = 0 + 1/2, \text{ and } \Lft_1(A1) = 2/3 + 0.
\]
\end{example}

\begin{remark}
Note that a natural length vector of $\subsdual$ is the vector governing the measure $\mu$ of $\subs$ as in equation \eqref{eq:evectmeasure}, giving geometric significance to the partitions as self-similar sets.
\end{remark}

Flow views are drawn in $\R \times [0,1]$, with sequences depicted in the horizontal direction as colored intervals. We build the {\bf \em level-$1$ flow view} as follows. Render each 1-supertile as the appropriate string of colored horizonal intervals thickened vertically by the measure of its cylinder set. This rendering is placed in the flow view so that it intersects $[0,1]$ at the interval of the same address. 
This represents the $1$-cylinder set of all tilings that have that supertile at the origin in that position, and its Lebesgue measure matches its $\mu$-measure. The period doubling level-$1$ flow view appears on the right of figure \ref{fig:2partion}. On the left is an  {\bf \em address diagram} that can be used to track the locations of the $n$th level partition elements in $[0,1]$.
\begin{figure}[ht]
\includegraphics[height=.175\textheight]{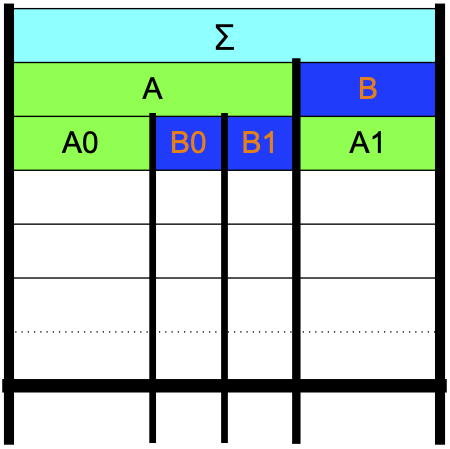}\hskip 3em
\includegraphics[height=.175\textheight]{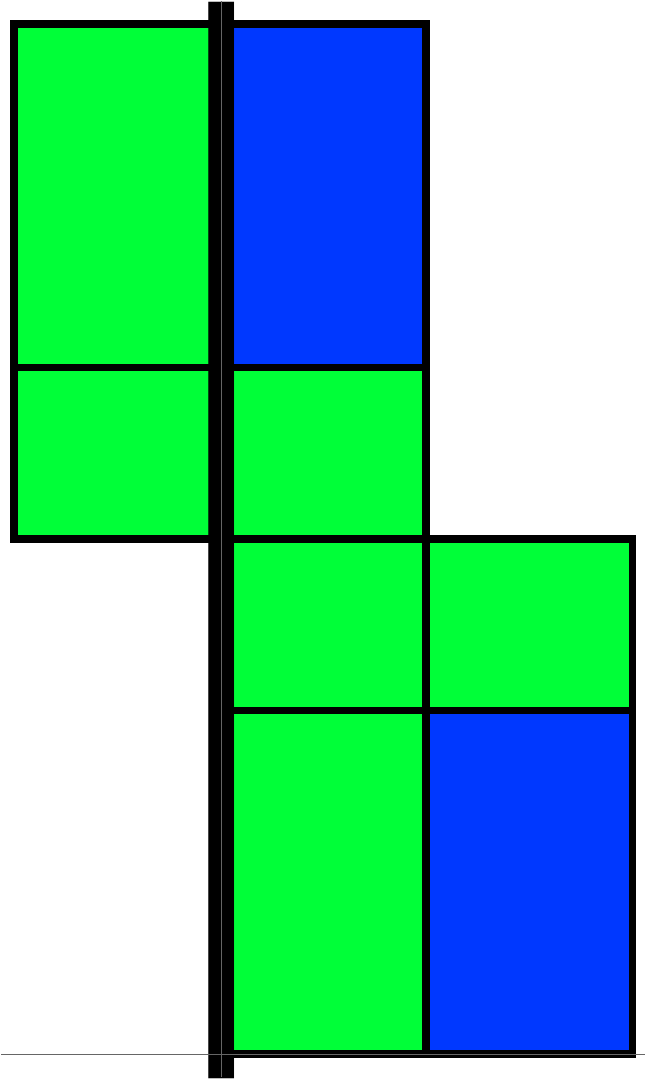}
\caption{The the level-1 address diagram and flow view for $\PD$.}
\label{fig:2partion}
\end{figure}

This process is used to refine intervals at level $(n-1)$ to level $n$ in general as follows. The refinement $\PartitionI_2$ is given by sets of the form $\II((\alet_1,j_1), (\alet_2, j_2))$, where $((\alet_1,j_1), (\alet_2, j_2)) \in \addr 2$.
We construct our refinement so that each $\II(\alet_1,j_1) \in \PartitionI_1$ is partitioned by the partition elements $\II((\alet_1,j_1), (\alet_2, j_2))$ for which $(\alet_2, j_2) \in\loc{\alet_1},$ placed in the order given by $\subsdual(\alet_1)$. 

Because $\mu$ is invariant, and because the 2-cylinder set $[\subs((\alet_1,j_1), (\alet_2, j_2))]$ is a  shift of the 2-cylinder $\subs^2([\alet_2])$ we have
\[
\mu([\subs(\alet_1,j_1)]) = \sum_{(\alet_2, j_2) \in\loc{\alet_1}} \mu([\subs((\alet_1,j_1), (\alet_2, j_2)) ])= \sum_{(\alet_2, j_2) \in\loc{\alet_1}} \mu([\alet_2])/\lambda^2.
\]

The interval $[0, \mu([\subs(\alet_1)]))$ is scaled by $1/\lambda$ from $ [0, \mu([\alet_1]))$ and so we use $\lft(\alet_2, j_2)/\lambda$ to partition it. This preserves the order given by $\subsdual$.  A scaled-down copy of the partition of $\II(\alet_1)$ is placed on every $\II(\alet_1, j_1)$. That is, take $\Lft_1(\subs(\alet_1,j_1))$ and add on $\lft(\alet_2, j_2)/\lambda$:
\[\Lft_2((\alet_1,j_1), (\alet_2, j_2)) = \Lft_1(\alet_1,j_1) + \lft(\alet_2, j_2) /\lambda = \Lft_0(\subs(\alet_1)(j_1)) + \lft(\alet_1, j_1) + \lft(\alet_2, j_2) /\lambda
\] 
We define $\II((\alet_1,j_1), (\alet_2, j_2))=[\Lft_2((\alet_1,j_1), (\alet_2, j_2)), \Lft_2((\alet_1,j_1), (\alet_2, j_2)) + \mu([\alet_2])/\lambda^2)$. 

To construct the {\bf \em level-2 flow view}, render each $2$-supertile  $ \subs((\alet_1,j_1), (\alet_2, j_2))$ horizontally as colored unit intervals with vertical thickness $\mu([\alet_2])/\lambda^2$, placed at $y=\Lft_2((\alet_1,j_1), (\alet_2, j_2))$.
See figure \ref{fig:2levelpartion} for the level-$2$ address diagram and flow view of the period-doubling substitution. Note that the level-$2$ flow view contains and refines the level-$1$ flow view.

\begin{figure}[ht]
\includegraphics[height=0.175\textheight]{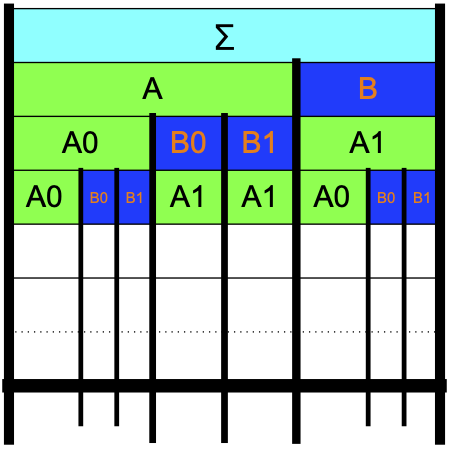}\hskip 3em
\includegraphics[height=0.175\textheight]{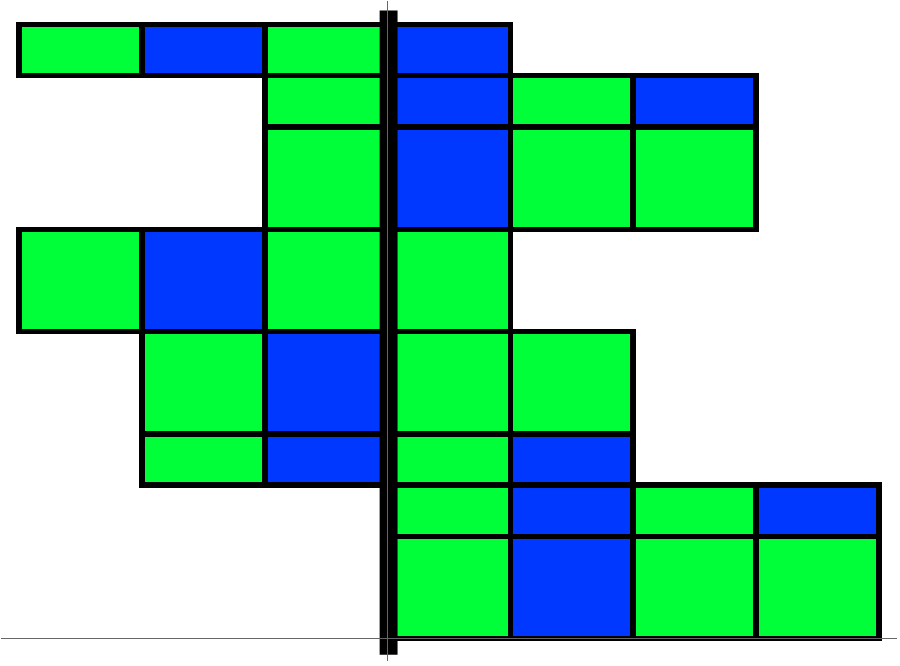}
\caption{The level-2 address diagram and flow view for $\PD$.}
\label{fig:2levelpartion}
\end{figure}

From here the refinements follow the same pattern and we can define the function $\Lft_n: \addr n \to [0,1)$ recursively or directly. 
If $\poslist{} = ((\alet_1,j_1),\ldots, (\alet_n, j_n))$, then
\begin{equation}
\Lft_n(\poslist{}) = \Lft_{n-1}( (\alet_1,j_1),\ldots, (\alet_{n-1}, j_{n-1})) + \frac{\lft(\alet_n,j_n)}{\lambda^{n-1}} =  \Lft_0(\subs(\alet_1)(j_1))\ +\sum_{k = 1}^n \frac{\lft(\alet_k, j_k)}{\lambda^{k-1}}.
\label{eqn:LE(p)}
\end{equation}
The interval corresponding to $\poslist{}$ is thus \[
 \II(\poslist{}) = [\Lft_n(\poslist{}), \Lft_n(\poslist{}) + \mu([\subs(\poslist{})])) = \left[\Lft_n(\poslist{}), \Lft_n(\poslist{}) + \frac{\mu([\alet_n])}{\lambda^n}\right)
,\] making the Lebesgue measure of $\II(\poslist{}) $ is equal to $\mu([\subs(\poslist{})])$.
We define the {\bf \em canonical partition sequence of $[0,1)$ given by $\subsdual$} to be $\left(\PartitionI_n\right)_{n = 0}^\infty,$ where
\[
\PartitionI_n= \{\II(\poslist{})  \text{ such that } \poslist{}\in \addr{n}\}.
\]

\begin{definition}\label{def:mtconj}
The {\bf \em canonical isomorphism given by $\subsdual$} is the map $\mtcong: \seqsp \to [0,1]$ given by \begin{equation}
\mtcong(\T) = \lim_{n \to \infty} \mtcong_n(\poslist n(\T))= \Lft_0(\subs(\alet_1)(j_1))+\sum_{k = 1}^\infty \frac{\lft(\alet_k, j_k)}{\lambda^{k-1}}, \text{ where } \poslist{}(\T) =  ((\alet_i, j_i))_{i = 1}^\infty.
\end{equation}
The {\bf \em (full) flow view given by $\subsdual$} is the graph of $\mtconj$, with each $\T \in \seqsp$ rendered as colored unit intervals drawn at $y=\mtconj(\T)$. 
\end{definition}

A high-level approximation of the full flow view of $\PD$ is shown on the left of figure \ref{fig:pdfinalstuff}.

\begin{proposition}
Let $\subs$ be a recognizable substitution whose subshift $(\seqsp, \mu)$ is minimal, and let $\mtcong: (\seqsp, \mu) \to ([0, 1], m)$ be a canonical isomorphism. Then $\mtcong$ is measure preserving, well-defined and uniformly continuous everywhere,  bijective almost everywhere, and at most $2|\ab|$:1.
 \label{Pf:mtcong}
\end{proposition}

\begin{proof}
We know  $\mtconj$ is well-defined because every tiling has a unique address and each infinite series converges by definition. 

For uniform continuity, given $\epsilon>0$, suppose  $\T$ and $\T'$ have the same address out to $N$, where $\epsilon > \lambda^{-N}$.  Then $\mtconj(\T)$ and $\mtconj(\T')$ map into the same element of $\PartitionI_N$, which has length smaller than $\lambda^{-N}$, and thus $|\mtconj(\T)-\mtconj(\T')|< \epsilon$.
There is a common recognizability radius $R_N$  to determine the $N$th supertile at the origin for any element of $\seqsp$. To ensure $\T$ and $\T'$ agree on that supertile we need only that $d(\T, \T') < 1/R_N$. This proves uniform continuity. Moreover, since $\seqsp$ is compact its image must be also, so $1 \in \mtcong(\seqsp)$. Since the partitions refine to points, $\mtcong$ is a surjection.

Now suppose $\mtconj(\T) = \mtconj(\T')$. One way this can happen is if $\poslist{}(\T) = \poslist{}(\T')$, which can happen if both $\T$ and $\T'$ are made of two infinite-order supertiles with the origin in precisely the same location of one of them. In this case there are at most $|\ab|$ possible choices for the half of $\T'$ that differs from $\T$.

If $\poslist{}(\T) \neq \poslist{}(\T')$ but $\mtconj(\T) = \mtcong(\T')$, let $n$ be the smallest integer for which $\poslist n(\T) \neq \poslist n(\T')$. All of the partial sums are nondecreasing and so if $\poslist{n}(\T) \neq \poslist{n}(\T')$ then WLOG we may assume $\Lft_n(\poslist{n}(\T))<\Lft_n(\poslist{n}({\T'}))$. The remaining terms in the series for $\poslist{}(\T')$ must then be 0, so $\mtconj(\T') = \Lft(\poslist{n}(\T'))$. This set of {\em left endpoints} in $[0,1]$ comes up often enough to name it:
\begin{equation}
\label{eqn:lends}
\lends = \{\Lft_n(\poslist{}), \, \poslist{} \in \addr{n} \text{ and } n \in \N\}.
\end{equation}
The remaining terms that comprise $\mtconj(\T)$ must be the maximum possible within $\II(\poslist{n-1}(\T))$, and this is also unique. So if $x = \mtcong(\T) = \mtcong(\T')$,
and $x \in \lends$, then $\mtcong^{-1}(x)$ contains $\T$, $\T'$, and any elements of $\seqsp$ with the same address as $\T$ or as $\T'$, for a total of $2|\ab|$ possible preimages.
The tails of the addresses of both $\T$ and $\T'$ use only a subset of the full label set $\domain$. By lemma \ref{lem:meas0}, this is a null set for $\mu$. 
\end{proof}

By construction Lebesgue measure is the push-forward of $\mu$ under 
$\mtcong$ and so
\begin{corollary}
For all integrable $f: [0,1] \to \C$, $\int_0^1 f(x) dm = \int_\seqsp f(\mtcong(\T)) d \mu$.
\label{lem:cov}
\end{corollary}

At points where $\mtcong$ is one-to-one its inverse is continuous in the following sense. 
\begin{corollary}
Let $x_0 \in [0,1)\setminus\lends$.
For every $\delta > 0$ there exists an $\epsilon'>0$ such that if $|x - x_0| < \epsilon'$, then $d(\mtconj^{-1}(x),\mtcong^{-1}(x_0))< \delta$ for any element of $\mtconj^{-1}(x)$.
\label{cor:inv.one.to.one}
\end{corollary}
\begin{proof}
Since $x_0 \notin \lends$ there is a unique $\T_0$ with $\mtconj(\T_0) = x_0$. The address of $\T_0$ has infinitely many nonminimal or nonmaximal elements, and so there is an $N$ such that the $N$-supertile in $\T_0$ contains $-M$ and $M$ in its domain for some $M > 1/\delta$. Fix such an $N$ and choose $\epsilon'>0$ such that $(x_0-\epsilon',x_0+\epsilon') \subset \II(\poslist{N}(\T_0))$. If $\T$ is such that $\mtconj(\T) \in B_{\epsilon'}(x_0)$, then $\poslist{N}(\T) = \poslist{N}(\T_0)$. This means that $\T$ and $\T'$ have the same $N$-supertile at the origin, and thus $d(\T, \T') < \delta$.
\end{proof}

\begin{example}
Figure \ref{FlowViewHeit} shows the flow view for the substitution 
$A \mapsto ACB, B \mapsto BCA, C\mapsto CAC$, with $A, B,$ and $C$ represented by green, blue, and red intervals respectively. This substitution is constant-length with expansion $\lambda = 3$ and has a period-2 substructure known as a `height' of 2 \cite{Host}. 
\begin{figure}[ht]
\includegraphics[width=5in]{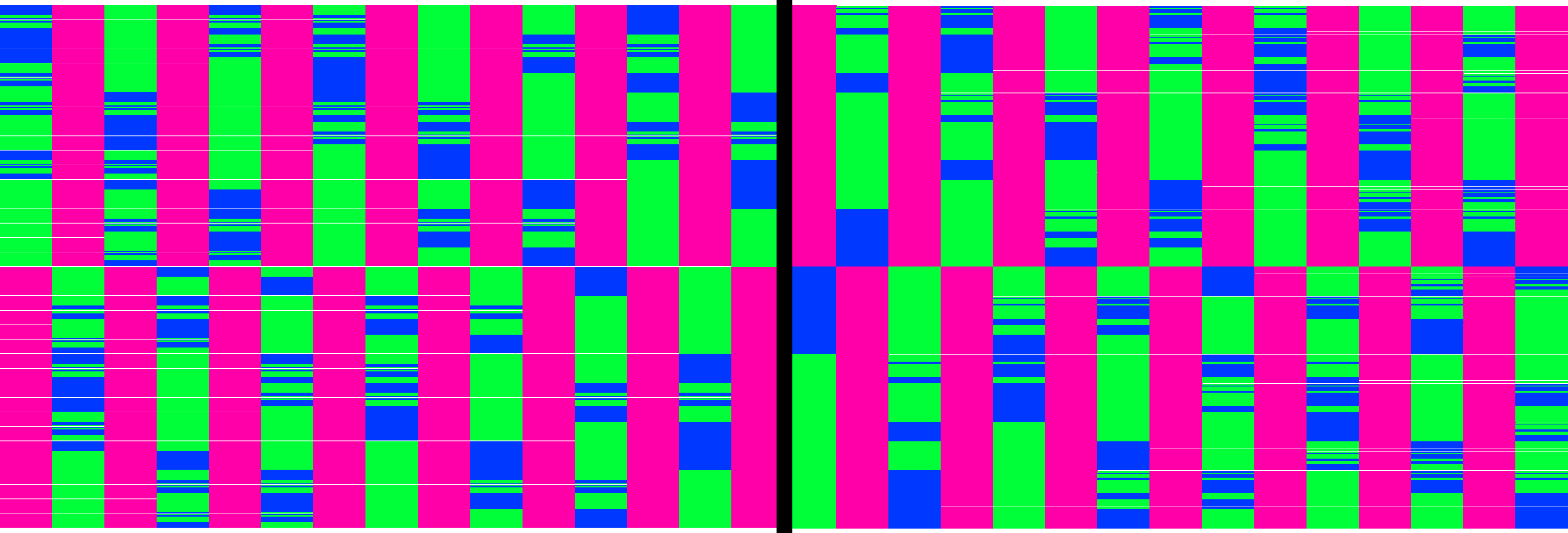}
\caption{High-level flow view of a substitution of nontrivial `height' 2 \cite{Host}.}
\label{FlowViewHeit}
\end{figure}
\end{example}

\section{The infinite interval exchange transformation}
\label{sec:IIET}
We return to the set of infinite addresses $\addr{\infty}$ to aid in our definition of the map $\iiet: [0,1] \to [0,1]$ that conjugates with the shift action via $\mtcong$.
Shifts in $\seqsp$ cause changes in the addresses in a way that is captured almost everywhere by a map\footnote{sometimes known as an {\em adic} or a {\em Vershik} map} on a stationary Bratteli diagram. It is already interesting to consider this system, but we need only the map, and solely as a bookkeeping device. For the interested, we briefly describe the Bratteli diagram first.

As is customary, there is one vertex at the top of the Bratteli diagram, and each level after that has vertex set $\domain$. The edges from a vertex $(\alet, j)$ at level $n$ are to the vertices $(\blet, k) \in \loc{\alet}$ at level $n+1$. Alternatively, the edges into $(\alet, j)$ at level $n$ from level $n-1$ connect to each $(\blet, k)$ for which $\subs(\alet)(j) = \blet$. We call the edge set $\edges$. The canonical partial order on each level is $(\alet, j) < (\alet,j+1)$ where $0 \le j < |\subs(\alet)|-1$. This relates the minimal and maximal paths in the Bratteli diagram to the the rightmost and leftmost positions in the domains of the supertiles.
Together $(\domain,\edges)$ form a stationary Bratteli diagram whose path space is given by $\addr{\infty}$. 

\begin{example}
For the period-doubling example, the vertices at each level are labeled by $\{A0, A1, B0, B1\}$ and $\loc{A} = \{A0, B0, B1\}$ and $\loc{B}=\{A1\}$. Here are the first few levels of the diagram.
\centerline{\includegraphics[width=.35\textwidth]{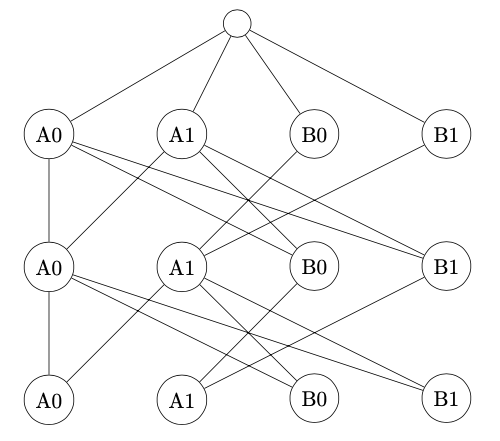}}
\end{example}

There are exactly $\numa$ minimal and $\numa$ maximal paths in each $\addr n$ that address the leftmost and rightmost positions of the $n$-supertile of each type, and we denote these by  
$\plmin{n}{\alet}$ and $\plmax{n}{\alet}, \alet \in \ab$.  For any $\poslist{} = ((\alet_1, j_1), (\alet_2, j_2), ...) \in \addr{}$ define 
$\enn(\poslist{})$ to be the first index at which an element of $\poslist{}$ can be increased, i.e. the smallest $k$ for which $\poslist{}[1,...k] \neq \plmax{k}{\alet_k}$. If $\enn(\poslist{}) = N$, then 
\[\poslist{}= (\plmax{N-1}{\subs(\alet_N)(j_N)}, (\alet_N, j_N),  (\alet_{N+1},j_{N+1}), ...),
\] where we note that $\subs(\alet_N)(j_N)$ is the type of the $N-1$ supertile in position $j_N$ of the $N$-supertile of type $\alet_N$.

\begin{definition}
The map $\vvv: \addr{} \to \addr{}$ is defined for any\[\poslist{} = \left(\plmax{N-1}{\subs(\alet_N)(j_N)}, (\alet_N, j_N),  (\alet_{N+1},j_{N+1}), \ldots\right)\] with $j_N < |\subs(\alet_N)| - 1$ to be
\begin{equation}
\vvv(\poslist{}) = \left(\plmin{N-1}{\subs(\alet_N)(j_N+1)}, (\alet_N, j_N+ 1),   (\alet_{N+1}, j_{N+1}), \ldots\right).
\label{eqn:vershikmap}
\end{equation}
The map $\vvv$ is not defined for any $\poslist{}$ with $\enn(\poslist{}) = \infty$.
\end{definition}

Here is how $\vvv$ keeps track of the address $\poslist{}(\T)$ as $\T$ is shifted. We define $\enn(\T) = \enn(\poslist{} (\T))$.
There is no change to the address of any $k$-supertile where $k>\enn(\T)$, since the boundary being crossed over is in its interior. The address of the $\enn(\T)$-supertile is increased by one but the supertile type does not change. All the supertiles of levels $k<\enn(\T)$ are now at the left of their domains in the minimal position. The types of those subtiles depends on $\subs(\alet_N)(j_N+1)$.

Note that $\poslist{\enn(\T)}(\T)$ is the address of the lowest-level supertile whose $n$-cylinder contains both $\T$ and $\shift(\T)$. That means $\poslist{}(\T)(k) = \poslist{}(\shift(\T))(k)$ for all $k > \enn(\T)$ and we say that $\poslist{}(\T)$ and $\poslist{}(\shift(\T))$ are  {\bf \em tail equivalent}.

\begin{lemma}
For all $\T \in \seqsp$ with $\enn(\T) < \infty$, 
$\poslist{}(\shift(\T)) = \vvv(\poslist{}(\T)).$
\label{lem:shiftvershikcommute}
\end{lemma}

\begin{proof}
Let $\enn(\T)=N$ so that the origin is situated at the rightmost end of all of $\T$'s $k$-supertiles for $k = 1, ..., N-1$. Write $\poslist{}(\T) = (\plmax{N-1}{\subs(\alet_N)(j_N))}, (\alet_N, j_N),  (\alet_{N+1}, j_{N+1}), \ldots)$. 

Shifting $\T$ moves the origin to the first element of the next $(N-1)$-supertile inside $\subs^{N}(\alet_N)$, which has type $\subs(\alet_N)(j_N+1)$. Now the origin in $\shift(\T)$ is in the leftmost position of the $(N-1)$-supertile of type $\subs(\alet_N)(j_N+1)$, so
$\poslist{N-1}(\shift(\T)) = \plmin{N-1}{\subs(\alet_N)(j_N+1)}$.  None of the addresses of supertiles of larger order than $N$ are altered. This means
\[ \poslist{}(\shift(\T)) = \left(\plmin{N-1}{\subs(\alet_N)(j_N+1)}, (\alet_N, j_N+1), (\alet_{N+1}, j_{N+1}), \ldots \right). \]
By equation \eqref{eqn:vershikmap} this is equal to $\vvv(\poslist{}(\T))$.
\end{proof}

In order to define $\iiet$ as a function that commutes with $\shift$ we need to give a unique address to almost every $x \in [0,1]$. That requires identifying the partition element that contains $x$ at each level. This is not a problem for the points on which $\mtcong$ is invertible, in which case we define $\poslist{}(x) =\poslist{}(\mtcong^{-1}(x))$. 

If $\mtconj$ is not one-to-one at $x$, Corollary \ref{cor:inv.one.to.one} implies that $x$ is the left endpoint of a partition element. Thus there is a $\T$ for which $x = \mtcong(\T) = \mtcong_K(\T)$ as a finite sum. We choose this preimage and define $\poslist{}(x) = \poslist{}(\T)$ and $\enn(x) = \enn(\T)$ for this $\T$. This allows $\iiet$ to take $x$ along with the elements in the interval above it. 

\begin{definition}\label{def:iiet}
 For $x \in [0,1]$ with $\enn(x) = N < \infty$, define the {\bf \em infinite interval exchange transformation given by $\subsdual$} to be
\begin{equation}
\iiet(x) = x - \Lft_N(\poslist N(x)) + \Lft_N(\vvv(\poslist N(x))).
\label{eqn:subsiiet}
\end{equation}
\end{definition}

\begin{theorem}\label{thm:iiet}
Let $\subs$ be a recognizable substitution with minimal subshift and let $\mtconj: (\seqsp, \mu) \to ([0,1], m)$ be a canonical isomorphism given by $\subsdual$. The IIET $\iiet$ given by $\subsdual$ is defined almost everywhere, and $\mtconj$ is a measurable conjugacy between $(\seqsp,\shift, \mu)$ and $([0,1],\iiet,m)$.
\end{theorem}

\begin{proof}
There are finitely many points at which $\iiet$ fails to be defined. Since $x=1$ is never in any finite partition interval, $\iiet$ is not defined there. There are also $\numa$ infinite maximal addresses representing an infinite-order supertile with domain $(-\infty, ...,-1, 0]$, and none of these have an image under $\vvv$. 
The image under $\mtconj$ of these sequences may not have a well-defined image under $\iiet$. (It might, depending on whether it is in $\lends$, in which case the IIET will track only one of the possible orbits.)
At every other $x \in [0,1]$ we know that $\vvv$ is well-defined on a well-defined address $\poslist{}(x)$, and for these $x$, $\iiet(x)$ is well defined.

Let $\T \in \seqsp$ with $\poslist{}(\T)= \left((\alet_n, j_n)\right)_{n = 1}^\infty$ and $\enn(\T) = N < \infty$ so that $\mtconj(\T)$ lies in $\II(\poslist{N}(\T))$.
We can write $\mtconj(\T) = \Lft_N(\poslist{N}(\T)) +  \sum_{n = N+1}^\infty \frac{\lft(\alet_n,j_n)}{\lambda^{n-1}}$. We have
\begin{align*}
\iiet(\mtconj(\T)) &=\mtconj(\T) - \Lft_N(\poslist{N}(\T))+\Lft_N(\vvv(\poslist{N}(\T)))\\
&= \left(\Lft(\poslist{N}(\T)) +  \sum_{n = N+1}^\infty \frac{\lft(\alet_n,j_n)}{\lambda^{n-1}}\right)  - \Lft_N(\poslist{N}(\T))+ \Lft_N(\vvv(\poslist{N}(\T)))\\
&=\Lft_N(\vvv(\poslist{N}(\T)))+  \sum_{n = N+1}^\infty \frac{\lft(\alet_n, j_n)}{\lambda^{n-1}}\\
&=\Lft_N(\poslist{N}(\shift(\T)))+  \sum_{n = N+1}^\infty \frac{\lft(\alet_n, j_n)}{\lambda^{n-1}} =\mtconj(\shift(\T)) ,
\end{align*}
with the last two equalities following from lemma \ref{lem:shiftvershikcommute} and the fact that $\T$ and $\shift(\T)$ are tail equivalent.
\end{proof}

\begin{corollary}\label{cor:iietest}
For any $n \in \N$ there is an exchange of $n(|\domain| - |\ab|) + |\ab|$ 
intervals that is equal to $\iiet$ on all but $|\ab|$ intervals of total measure $\le \lambda^{-n}$.
\end{corollary}

\begin{proof} Fix an $n$ and consider $x \in [0,1]$. If $\enn(x) \le n$ let 
$\iiet_n(x) = \iiet(x)$. If $\enn(x) > n$ then it must be that $x \in \II(\plmax{n}{\alet})$ for some $\alet \in \ab$. Since $\mu([\subs(\plmax{n}{\alet}]]) = \mu([\subs(\plmin{n}{\alet})]) = \mu([\alet])/\lambda^n$ we define $\iiet_n(x) = x - \Lft(\plmax{n}{\alet})+\Lft(\plmin{n}{\alet})$. This temporarily fills in what happens at the ends of $n$-supertiles by sending them to the start of $n$-supertiles of the same type. The total length of the intervals on which $\iiet$ and $\iiet_n$ have the potential to differ is $\sum_{\alet \in \ab} \mu([\subs^n(\alet)]) = 1/\lambda^n$.

We count the number of intervals needed for $\iiet_n$ inductively. There are exactly $\numa$ maximal addresses and $\numa$ minimal addresses, one per element of $\ab$. To make $\iiet_1$, there are $|\domain|$ total intervals and $\numa$ of them have maximal addresses. On the nonmaximal intervals, of which there are $|\domain| - \numa$, $\iiet$ and $\iiet_1$ agree. On the remaining $\numa$ intervals they may disagree.

There are $|\domain| - \numa$ new nonmaximal intervals on which $\iiet_2$ and $\iiet$ agree that come from refining the maximal partition elements in $\PartitionI_1$. Thus $\iiet_2$ and $\iiet$ agree on $2(|\domain| - \numa)$ intervals and potentially disagree on $\numa$ intervals.

At each stage the function $\iiet_n$ can be thought of as refining $\iiet_{n-1}$ by filling in what happens to the $|\domain| - \numa$ new nonmaximal intervals that appear as the maximal elements of $\PartitionI_{n-1}$ are refined. The potentially disagreeing intervals are smaller at each stage by a factor of $1/\lambda$.

\end{proof}

Figure \ref{PDapproxiiet} shows a progression of the first three approximants for the period-doubling IIET. The blue is the part in agreement and the orange is the part that needs refinement. The vertical lines that extend to 0 are artifacts from the software used to generate the image and should be disregarded.
\begin{figure}[ht]
\includegraphics[width = 5.5in]{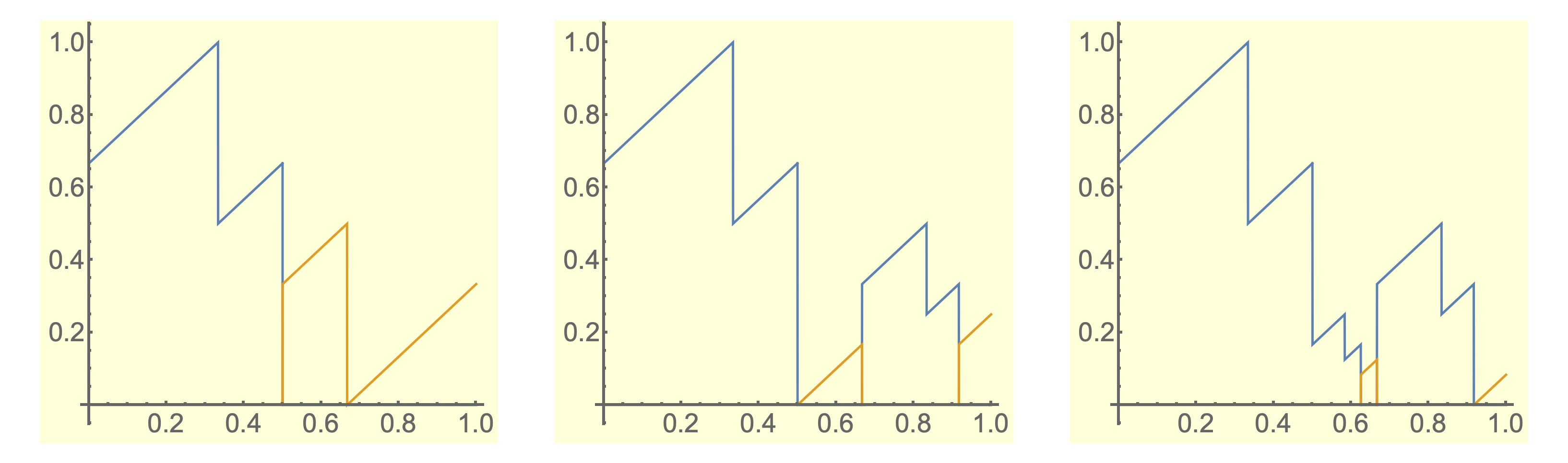}
\caption{The graphs of $\iiet_1, \iiet_2,$ and $\iiet_3$ for $\PD$.}
\label{PDapproxiiet}
\end{figure}
Although $\iiet_3$ is only accurate on $7/8$ of the interval, the pattern for filling in the remainder is visible from the progression. Moreover the refinements appear to provide the key to understanding the form of self-similarity taken. High-resolution approximations for the full flow view and IIET for the period-doubling substitution appear in figure \ref{fig:pdfinalstuff}.

\begin{figure}[ht]
\centering
\includegraphics[width=.9\textwidth]{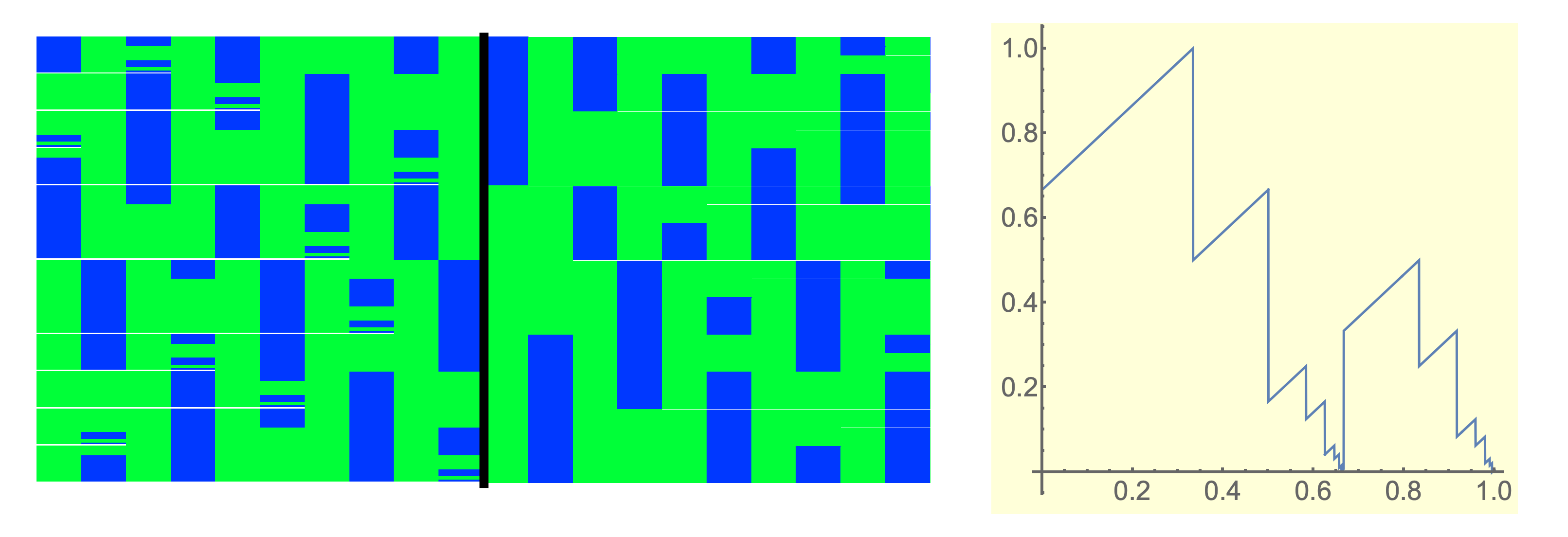}
\caption{High-level approximations to the flow view and IIET of the period-doubling substitution. The vertical lines in the IIET connect the ends of jump discontinuities.}
\label{fig:pdfinalstuff}
\end{figure}

\subsection{Self-similarity of some $\iiet$}

All of the IIETs shown in this document exhibit repetitive properties that appear to be a form of self-similarity, perhaps via a graph-directed IFS. This reveals geometrically the self-inducing structure of $\seqsp$. Here is a case where we can prove a specific form of self-similarity.
\begin{definition}
We call a substitution {\bf \em proper} if there are $\blet, \clet \in \ab$ such that $\subs(\alet)$ begins with $\blet$ and ends with $\clet$ for all $\alet \in \ab$. That is, for each $\alet \in \ab, \subs(\alet)(0) = \blet$ and $\subs(\alet)(|\subs(\alet)|-1)=\clet$.
\label{def:proper}
\end{definition}

\begin{proposition}\label{prop:self-similarilty}
Let $\subs$ be a recognizable and proper substitution with minimal subshift, and suppose $|\subs(\alet)|>1$ for each $\alet \in \ab$.
Then $\subs$ has a canonical IIET $\iiet$ for which there is a constant $\kappa \in [0,1)$ such that\begin{equation}
\iiet(x) = \lambda(\iiet(x/\lambda) + \kappa) \text{ for a.e. } x \in [0,1].
\label{eqn:selfsimiietprop}
\end{equation}
\end{proposition}

\begin{proof}
Suppose every substituted letter begins with $\blet$ and ends with $\clet$ and construct the initial partition $\PartitionI_0$ so that $\II(\clet) = [0, \mu([\clet]))$. No other restrictions on $\mtcong_0$ and $\PartitionI_0$ are required.
 
Next, we need to choose $\lft$ and $\subsdual$.
Since every supertile ends in a $\clet$, we know $\subsdual(\clet)$ contains all the letters in $\ab$ at least once. That immediately tells us that $\mu(\subs([\clet]))\ge 1/\lambda$. Allow the prefix of $\subsdual(\clet)$ be given by the letters of $\ab$ in whatever order was chosen for $\PartitionI_0$ so that $[0, 1/\lambda)$ is partitioned the same way as $\PartitionI_0$ except scaled by $1/\lambda$. We will use the interval $[0, 1/\lambda)$ to be associated to the maximal address of each 1-supertile.

By similar logic we know that $\subsdual(\blet)$ contains all of the letters of  $\ab$ at least once
and so we include a copy of $\PartitionI_0$ scaled by $\lambda$ in $\II(\blet)$ for the supertiles with minimal addresses. Denote this copy of $\PartitionI_0$ as $[\kappa, \kappa + 1/\lambda)$ and note that $\kappa$ can be chosen to be in $[1/\lambda, 1)$ since $|\subs(\alet)|>1$ for each $\alet \in \ab$. These choices constrain a portion of $\subsdual(\clet)$ and $\subsdual(\blet)$; there are no restrictions on the rest of $\subsdual$.
Consider the canonical isomorphism $\mtcong$ and IIET $\iiet$  for this choice of initial partition and refinement maps. We show that $\iiet$ is self-similar.

The square $Q=[0, 1/\lambda) \times [\kappa, \kappa + 1/\lambda)$ contains all transitions from maximal 1-supertiles to minimal 1-supertiles and is a scaled-down copy of the graph of $\iiet$ on all of $[0,1] \times [0,1]$. Any $x \in [0, 1/\lambda)$ has $\enn(x) \ge 1$ and corresponds to the 1-supertiles with maximal addresses (``maximal 1-supertiles", for short). The maximal 1-supertiles that are not maximal 2-supertiles correspond to values of $x$ in the interval $[1/\lambda^2, 1/\lambda)$. When the 2-supertile corresponding to $x$ is shifted, the maximal 1-supertile at 0 becomes a minimal 1-supertile with the same transitions as at the first level. 
That means that if $x \in [1/\lambda^2, 1/\lambda)$, then $\iiet(x) = \lambda^{-1} \iiet(\lambda x)+ \kappa$. Alternatively, if $x \in [1/\lambda, 1)$, then $\iiet(x) = \lambda(\iiet(x/\lambda) - \kappa)$.

We extend the result to  $[1/\lambda^2, 1/\lambda)$ next. The interval $[0, 1/\lambda^2)$ contains all $x$ for which $\enn(x) \ge 2$; it corresponds to the 2-supertiles with maximal addresses. It must therefore be mapped 
to the interval of length $1/\lambda^2$ corresponding to the 2-supertiles with minimal addresses, which begins at $\kappa(1 + 1/\lambda)$. The transitions between non-maximal 3-supertiles are the same as the transitions between non-maximal tiles within their supertiles, so \[\iiet(x) =\lambda^{-2} \iiet(\lambda^2x) + \kappa(1+1/\lambda).\]
Since $\lambda^2 x \in [1/\lambda, 1)$, that means 
\[
\iiet(\lambda^2 x)= \lambda(\iiet(\lambda^2 x/\lambda) - \kappa)=\lambda(\iiet(\lambda x) -\kappa).
\]
Plugging into the expression for $\iiet(x)$ yields 
\[\iiet(x) = \lambda^{-2}\left(\lambda(\iiet(\lambda x) -\kappa)\right) + \kappa(1+1/\lambda)=\lambda^{-1}\iiet(\lambda x) + \kappa.
\]
Since $\lambda x \in [1/\lambda^2, 1/\lambda)$ this extends the result \eqref{eqn:selfsimiietprop} to $[1/\lambda^2, 1)$. An analogous induction argument will find the result to hold for all intervals of the form $[1/\lambda^n, 1)$. 
\end{proof}

\begin{remark}
Examples satisfying the proposition can be made with many known ergodic properties. The one shown in figure \ref{fig:endsame} has purely discrete dynamical spectrum.
\end{remark}

\begin{figure}[ht]
\centering
\includegraphics[width=.3\textwidth]{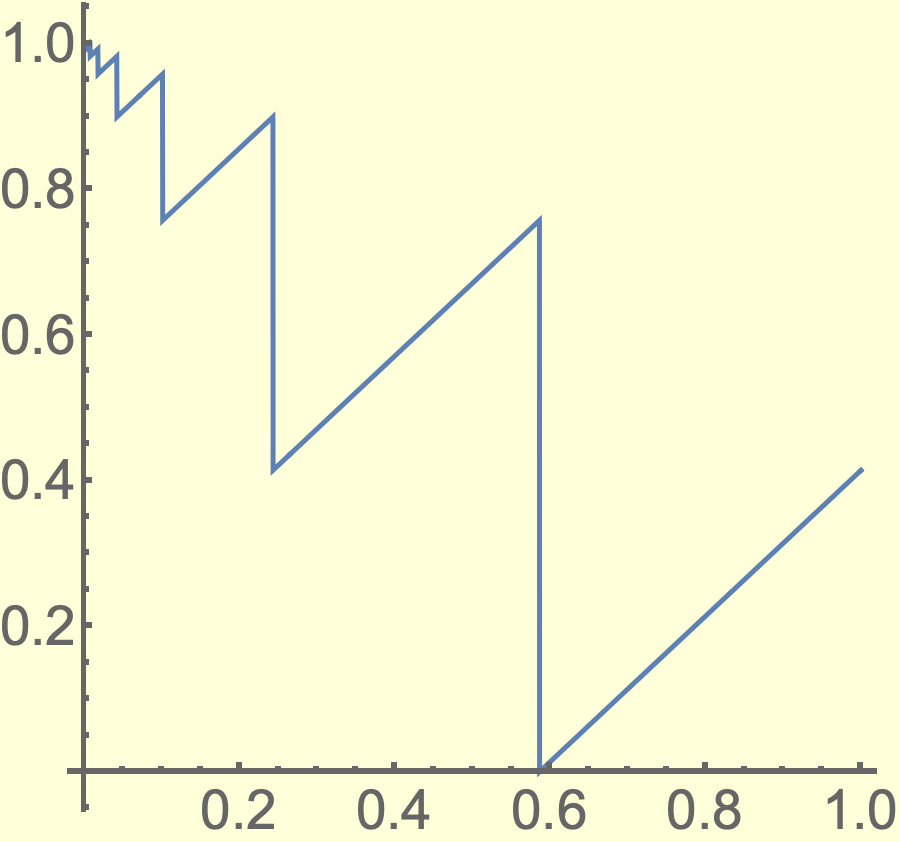}
\caption{Shazam! A self-similar IIET for $A \mapsto BBA, \, B \mapsto BA$.}
\label{fig:endsame}
\end{figure}

\section{Spectral analysis of $(\seqsp,\shift,\mu)$}
\label{sec:spectral}

\subsection{Review: spectral analysis in a general system (see e.g. \cite{Petersen.book})}
One way to study the behavior of a dynamical system $(X, T, \nu)$ is through the measurable functions it supports. It is particularly convenient to consider $L^2(X, \nu)$ because it is a Hilbert space. 
The ``dynamical spectrum" is the spectrum of the so-called Koopman operator $U$ on $L^2(X, \nu)$ given by $U(h) =  h\circ T$. It is natural to ask what happens under powers of the operator, which leads to the Fourier analysis approach we describe here. 

The {\bf \em spectral coefficient} $\hat{h}(j)$ is defined as the inner product
\[\hat{h}(j) = \left\langle U^j(h), \, h \right\rangle= \int_X h(T^j(x))  \overline{h(x)}d \nu(x).\]
This expression compares the measurements $h$ takes at each pair $(x,T^j(x))$ and averages the result over all $x$. 
It is well-known that the sequence of spectral coefficients is positive definite and so there is a {\bf \em spectral measure} $\specmeas_h$ on the circle $S^1$ for which
\[\hat{h}(j) =\int_{S^1} z ^j d\specmeas_h \quad \text{ for all } \quad j \in \Z.
\]

An {\bf \em eigenfunction} for $T$ is an $h \in L^2(X, \nu)$
for which there exists an {\bf \em eigenvalue} $R \in \R$ with $h(T(x)) = R\, h(x)$ for all $x \in X$. All constant functions are eigenfunctions and so it is customary to consider $h$ for which $\int_X h \, d\nu = 0$ when looking at spectral measures.

\subsection{The spectral measure of $\mtcong$}
In $\mtcong$ we have an extraordinary function in $L^2(\seqsp,\mu)$. It the best `observable' there is: it measures the features of  $\T \in \seqsp$ so accurately that it can almost always place it in a unique location in $[0,1]$. Moreover, that location is close to other $\T$s that strongly `resemble' it as measured using any other test function in $L^2(X,\mu)$. We conjecture that this implies that the spectral measure of $\mtcong$ is the maximal spectral type of the system.

For each $j \in \Z$ the spectral coefficient is computed using Corollary \ref{lem:cov} to be
\[
\hat{\mtconj}(j) = \int_\seqsp \mtconj(\shift^j(\T)) \overline{\mtconj(\T)} \mu(d\T)    
= \int_0^1\, x \, \iiet^j(x) dx,
\]
and since $\iiet^k(x) = x + c$, where $c$ depends only on any sufficiently nonmaximal partition element containing $x$, the integrand is a piecewise sum of upward-facing quadratics $x^2 + c x$.

It is known \cite{Host} that shifting by sequences of the form $(|\subs^n(\alet)|)_{n=1}^\infty$ reveals the presence or absence of nonconstant eigenfunctions. Figures \ref{fig:2adiciietcomp} through \ref{fig:npfiietbigits} in the discussions below exhibit a variety of behaviors of  $\iiet^{|\subs^n(\alet)|}(x)$ for different substitutions, beginning with the constant-length case.

\subsection{Special case: constant-length substitutions}

We say $\subs$ is a substitution of {\bf \em constant length} if there is a $K\in \{2, 3, ...\}$ for which $|\subs(\alet)| = K$ for all $\alet \in \ab$.  In this case $K$ is the expansion factor and all supertiles are of length $K^n$. Our results may not surprise you  after looking at the comparison of relatively large powers of four examples with $K=2$ in figure \ref{fig:2adiciietcomp}.

\begin{figure}[ht]
\centering
\includegraphics[width=.9\textwidth]{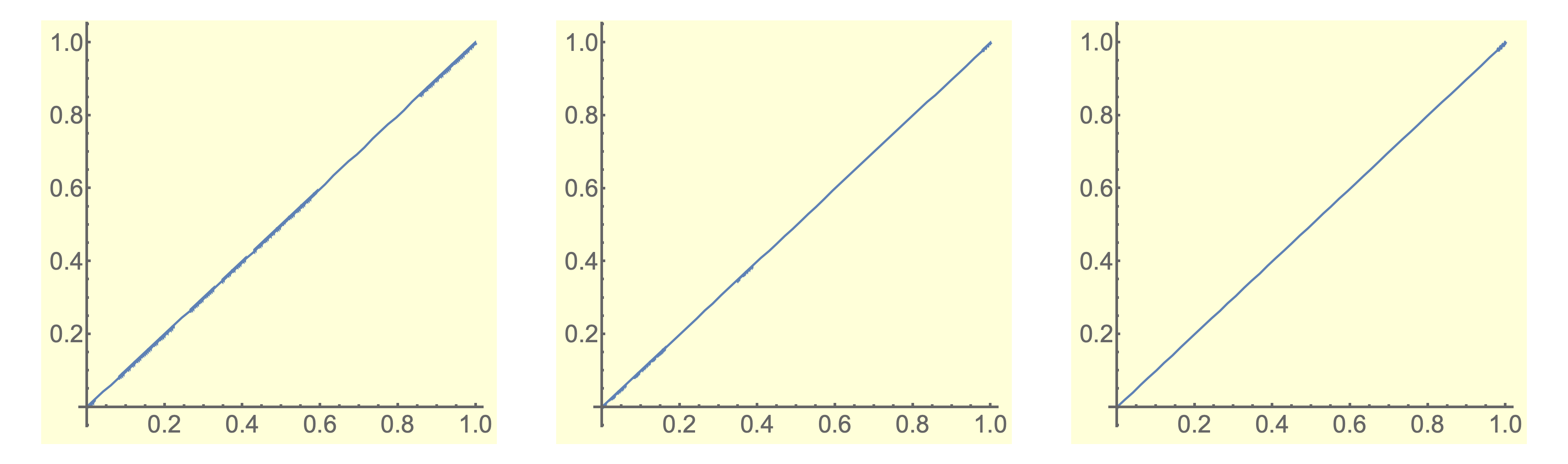}\\
\includegraphics[width=.9\textwidth]{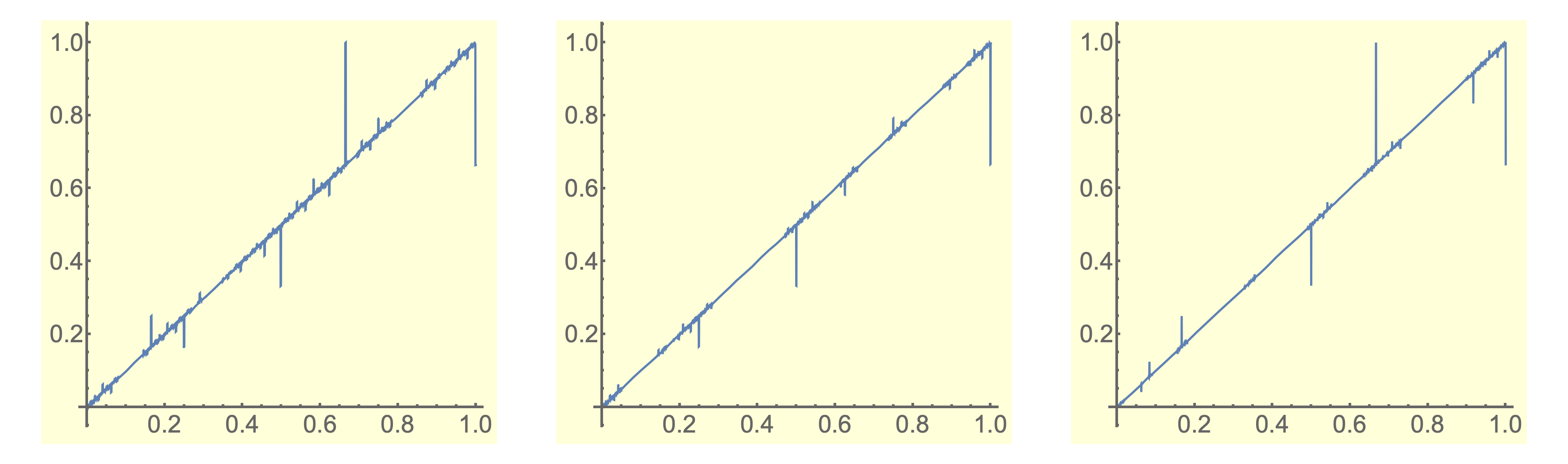}\\
\includegraphics[width=.9\textwidth]{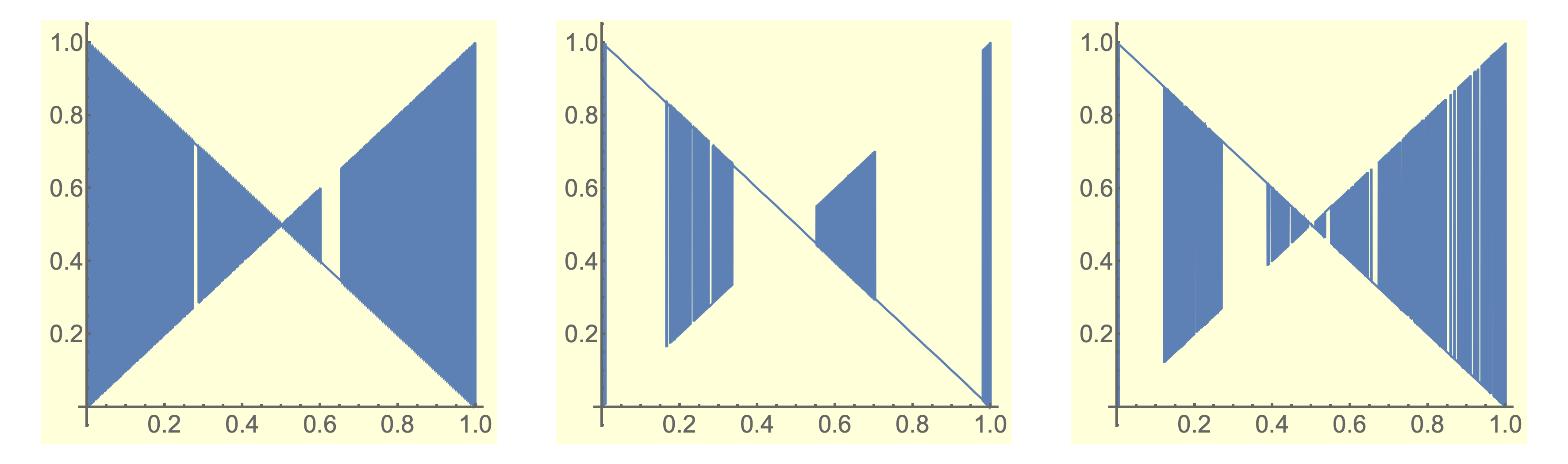}
\includegraphics[width=.9\textwidth]{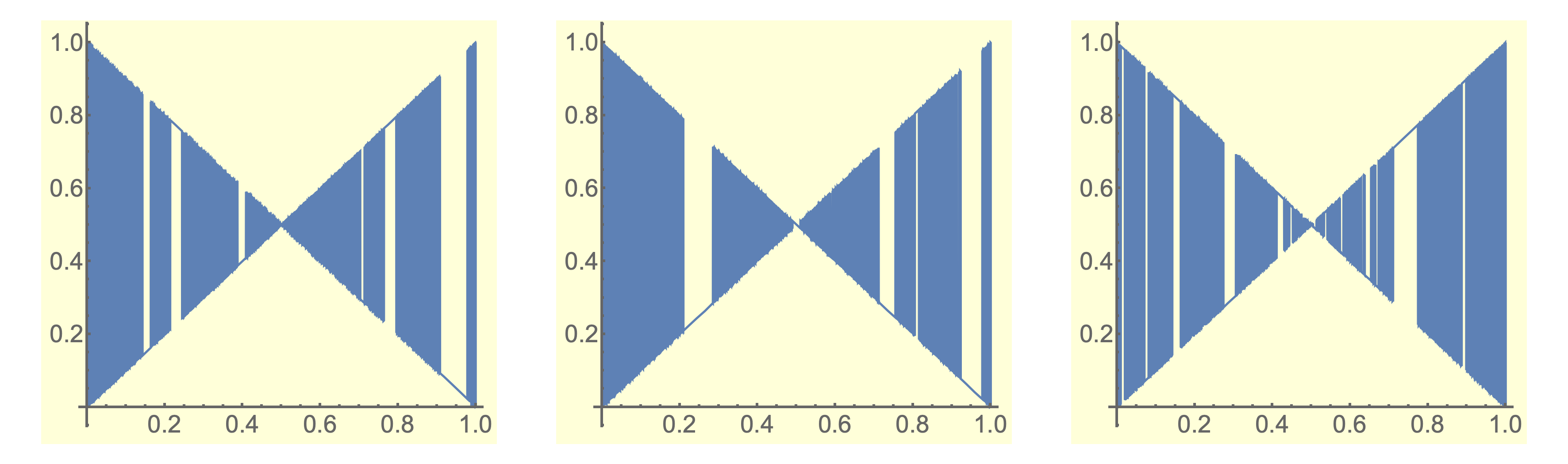}
\caption{Rows are the 2-adic odometer, period-doubling, Thue--Morse, and Rudin--Shapiro IIETs. The graphs are $\iiet_{20}^j$, where $j = 2^8, 2^9, $ and $2^{10}$.}
\label{fig:2adiciietcomp}
\end{figure}

\begin{proposition}\label{prop:constantlengthfinite}
If $\subs$ is a primitive recognizable constant-length substitution with expansion factor $K$ then $\left(\iiet^{K^n}(x)\right)_{n=1}^\infty$ has at most $|\ab|$ accumulation points for a.e. $x \in [0,1]$. 
\end{proposition}

\begin{proof}
Let $x$ be such that $\iiet^j(x)$ is defined for all $j \in \N$. Choose $\T \in \mtcong^{-1}(x)$ for which $\poslist{}(x) = \poslist{}(\T)$. Since $K^n$ is the length of an $n$-supertile, the $n$-supertile at the origin in $\shift^{K^n}(\T)$ and in $\T$ are the same position modulo $K^n$ but may be of different types. There are $|\ab|$ possible $n$-subintervals that could contain $\iiet^{K^n}(x)$.

For $N > n$, $\T(0)$ and $\shift^{K^N}(\T)(0)$ now are in the same position modulo $K^N$, so as before there are exactly $|\ab|$ subintervals that could contain $\mtconj(\shift^{K^N}(\T))$.  Shifting by $K^N$ matches up $n$-supertiles as well, so the possible intervals for $\T$ and $\shift^{K^N}(\T)$ are contained in the intervals for $\T$ and $\shift^{K^n}(\T)$. Thus there are $|\ab|$ nested sequences of intervals of vanishing length that $\iiet^{K^n}(x)$ can visit as $n \to \infty$. The limits of their left endpoints are the possible accumulation points for $\left\{\iiet^{K^n}(x) , \, n \in \N \right\}$.
\end{proof}

The spectrum of constant-length substitutions is well-understood (see \cite{Dekking} and the survey \cite{Fogg}). The spectrum always contains $\Z[1/K]$, representing the underlying $K$-adic odometer structure.
Questions of whether there are any other eigenfunctions, or any other significantly different functions at all, depends on how the letters populate the locations in $\subs^n(\alet)$ as $\alet$ varies. A fundamental notion is the following.

\begin{definition} A substitution $\subs$ of constant length $K > 1$ has a {\bf \em coincidence} if there are $N \in \N, \blet \in \ab$ and $\coinj \in \{0, \ldots, K^N-1\}$ such that $\blet= \subs^N(\alet)(\coinj)$ for all $\alet \in \ab$. Elements in $\domain$ of the form $(\alet, \coinj)$ as well as $\coinj$ itself are called {\bf \em coincidence addresses}, and we say that the substitution {\bf \em has a coincidence} if it has at least one coincidence address.
\end{definition}

\begin{theorem} Let $\subs$ be a primitive substitution of constant length $K>1$ with canonical isomorphism $\mtcong:(\seqsp, \shift, \mu) \to ([0, 1], \iiet, m)$. Then  $\lim_{n \to \infty}\iiet^{K^n}(x) =x$ almost everywhere if and only if $\subs$ has a coincidence.
\label{thm:coincthm}
\end{theorem}

\begin{proof}
Suppose there is a coincidence, which by taking powers if necessary can be assumed to occur
at $\coinj \in \{0, \ldots, K-1\}$. Consider $\seqsp_c$, the set of all $\T \in \seqsp$ such that $\poslist{}(\T)$ contains infinitely many coincidence addresses. This is a set of full measure by Lemma \ref{lem:meas0} since it contains $\seqsp_0$.

Let $x = \mtconj(\T)$ for $\T \in \seqsp_c$ and let $\epsilon>0$. Find some $n > 1+|\ln(\epsilon)|$ at which $\poslist{}(\T)$ has a coincidence label, i.e.
\[ 
\poslist{}(\T)=\left(\poslist{n-1}(\T), (\alet_n, j_c), (\alet_{n+1}, j_{n+1}), \ldots\right)
.\]
Then $\poslist{n-1}(\T)$ is independent of $\alet_n$ and depends only on the positions $j_k$ for $1 \le k \le n-1$. Any shift of $\T$ by $K^N$ where $N \ge n$ aligns the $n$-supertiles since $K^N \equiv 0 \pmod{K^n}$, so $\poslist{}(\shift^{K^N}(\T))$ also has a coincidence label at $n$. 
That is, $\T$ and $\shift^{K^N}(\T)$ have the same $n-1$ supertile with 0 in the same location. Their images under $\mtconj$ must both be in $\II(\poslist{n-1}(\T))$
which has length $= 1/K^{n-1} < \epsilon$. Since $\mtcong(\shift^{K^N}(\T)) = \iiet^{K^n}(\mtcong(\T)) \to \mtcong(\T)$ this shows $\iiet^{K^n}(x) \to x$ for any $x \in \mtcong(\seqsp_c)$.
 
Now suppose $\lim_{n \to \infty}\iiet^{K^n}(x) =x$ almost everywhere and let \[[\alet \blet] := \{\T \in \seqsp \, | \, \T(0) = \alet \text{ and } \T(1) = \blet\}.\]
Let $\delta>0$ such that $\delta < \mu([\alet \blet])$ for all legal two-letter words $\alet \blet: \{0, 1\} \to \ab$. By Egorov's theorem there is a set $\dd \subset [0,1]$ on which the convergence is uniform and $m(\dd) \ge 1-\delta$. 
Choose $M$ so that if $n \ge M$ then $|\iiet^{K^n}(x) - x| < \delta$ for all $x \in \dd$.
 
For all $j = 0, 1, \ldots, K^M -1$, let $\efull(j)$ be the set of all sequences in the $j$th spot of their $M$ supertile in the sense that \[\efull(j) = \bigcup_{\alet \in \ab} \shift^j(\subs^M([\alet)]).\]
This has measure $\mu(\efull(j)) = 1/K^M$.  Let $\ed = \mtconj^{-1}(\dd)$ and $\ed(j) = \efull(j) \cap \ed$, and suppose $\coinj$ is an index for which $\mu(\ed(j)) \le \mu(\ed(\coinj))$ for all $j \in \{0, ..., K^M-1\}$. 
 We claim $\coinj$ is a coincidence for $\subs$.
 
 For the sake of contradiction suppose there are $\hat{\alet}$  and $\hat{\blet}$ that differ at their $\coinj$th spot: $\subs^M(\hat{\alet})(\coinj) \neq \subs^M(\hat{\blet})(\coinj)$. Without loss of generality we may assume that $\hat{\alet}$ and $\hat{\blet}$ are chosen so that the word $\hat{\alet} \hat{\blet}$ is a legal word since if no such choice was possible then $\coinj$ is a coincidence and we are done.
 
Suppose $\T \in \shift^{\coinj}(\subs^M([\hat{\alet}\hat{\blet}]))$. Since $\T \in \shift^{\coinj}(\subs^M([\hat{\alet}]))$ and $\shift^{K^M}(\T)  \in \shift^{\coinj}(\subs^M([\hat{\blet}]))$ we know that $\mtconj(\T)$ and $\mtcong(\shift^{K^M}(\T))=\iiet^{K^M}(\mtcong(\T))$ are in disjoint subintervals of $\PartitionI_{M}$.
Thus $|\mtconj(\T)-\iiet^{K^M}(\mtcong(\T))| > \delta$. 
That means $\mu(\ed(\coinj)) \le \mu(\efull(\coinj)) - \mu(\shift^{\coinj}(\subs^M[\hat{\alet}\hat{\blet}]))= 1/K^M - \mu([\hat{\alet}\hat{\blet}]))/K^M.$ Since $\mu(\ed(\coinj))$ is the maximum measure over the $j$s we see that
 \[\mu(\ed) = \sum_{j = 1}^{K^M} \mu(\ed(j)) \le K^M \mu(\ed(\coinj))\le 1-\mu([\hat{\alet}\hat{\blet}]) < 1-\delta.\]
 Since $\mu(\ed) = m(\dd)$, this contradicts the definition of $\dd$ and completes the proof.
\end{proof}

\subsection{The non-constant length case.}
The author had originally posited a connection between having purely discrete spectrum and some type of convergence of $\iiet^{|\subs^n(\alet)|}(x)$ to the identity. This problem is closely connected to the {\em Pisot substitution conjecture} \cite{AkiyamaBargePisotSubstitutionConjecture}. The next few figures show shifts by large supertile amounts for some familiar examples: Fibonacci in figure \ref{fig:fibiietbigits}, tribonacci in figure \ref{fig:ariietbigits}, and two substitutions with weakly mixing subshifts in figures \ref{fig:Chaconiietbigits} and \ref{fig:npfiietbigits}.

\begin{figure}[ht]
\centering
\includegraphics[width=.9\textwidth]{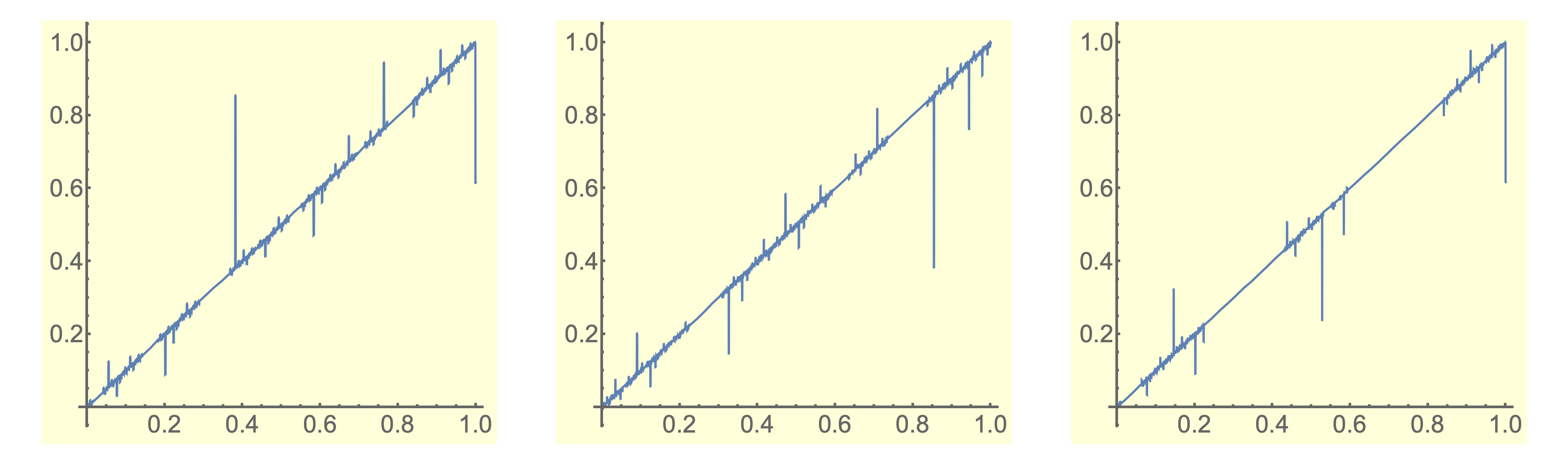}
\caption{Fibonacci IIET $\iiet_{30}^j$, where $j = 377, j = 610, $ and $j=987$.}
\label{fig:fibiietbigits}
\end{figure}

\begin{figure}[ht]
\centering
\includegraphics[width=.9\textwidth]{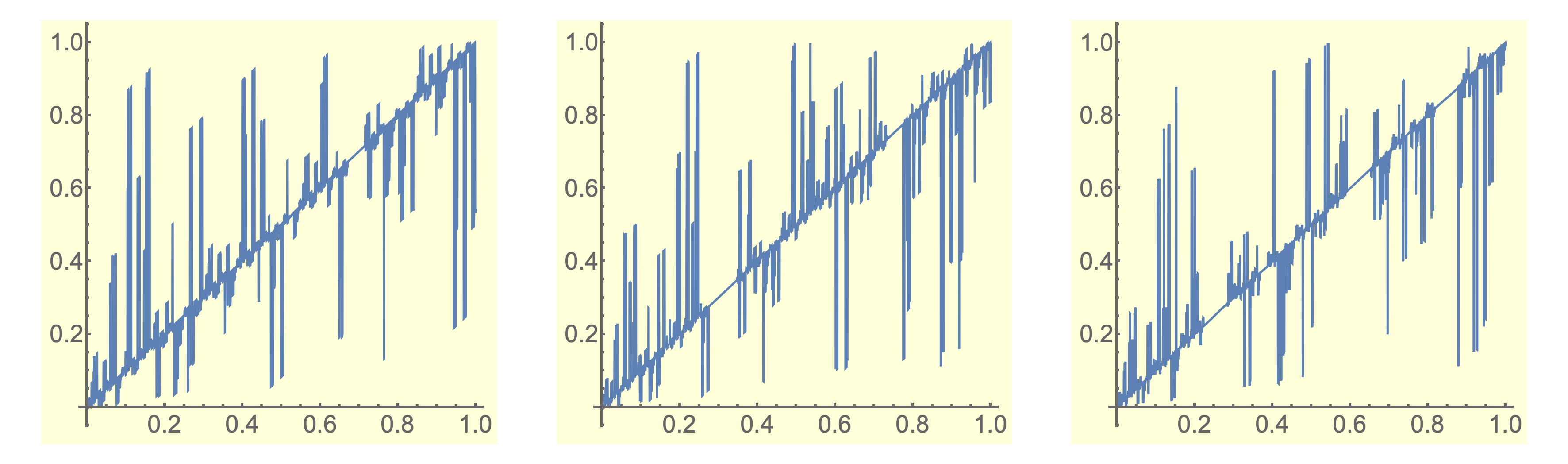}
\caption{The IIET $\iiet_{20}^j$ for the {\em tribonacci} substitution $A\mapsto AB, \,  B \mapsto AC, \, C \mapsto A$, for $j =204, j = 574, $ and $j = 927$.}
\label{fig:ariietbigits}
\end{figure}

\begin{figure}[ht]
\centering
\includegraphics[width=.9\textwidth]{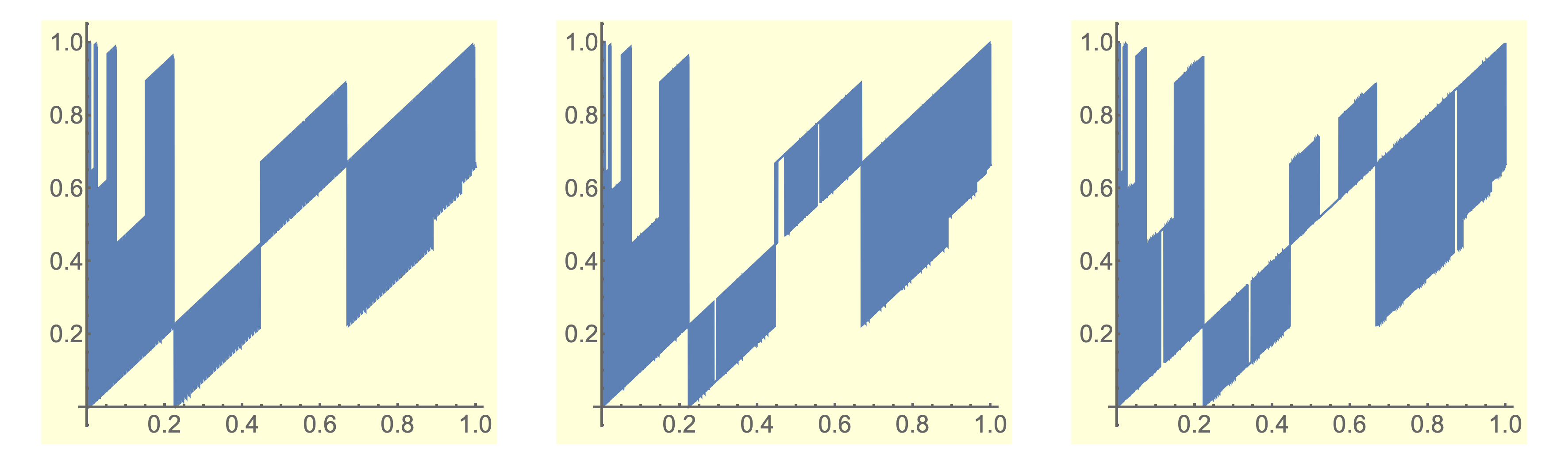}
\caption{The IIET $\iiet_{20}^j$ for the Chacon substitution $A \mapsto AABA$ and $B \mapsto B$, where $j = 121, j = 364,$ and $j = 1093$. The subshift is uniquely ergodic and weakly mixing.}
\label{fig:Chaconiietbigits}
\end{figure}

\begin{figure}[ht]
\centering
\includegraphics[width=.9\textwidth]{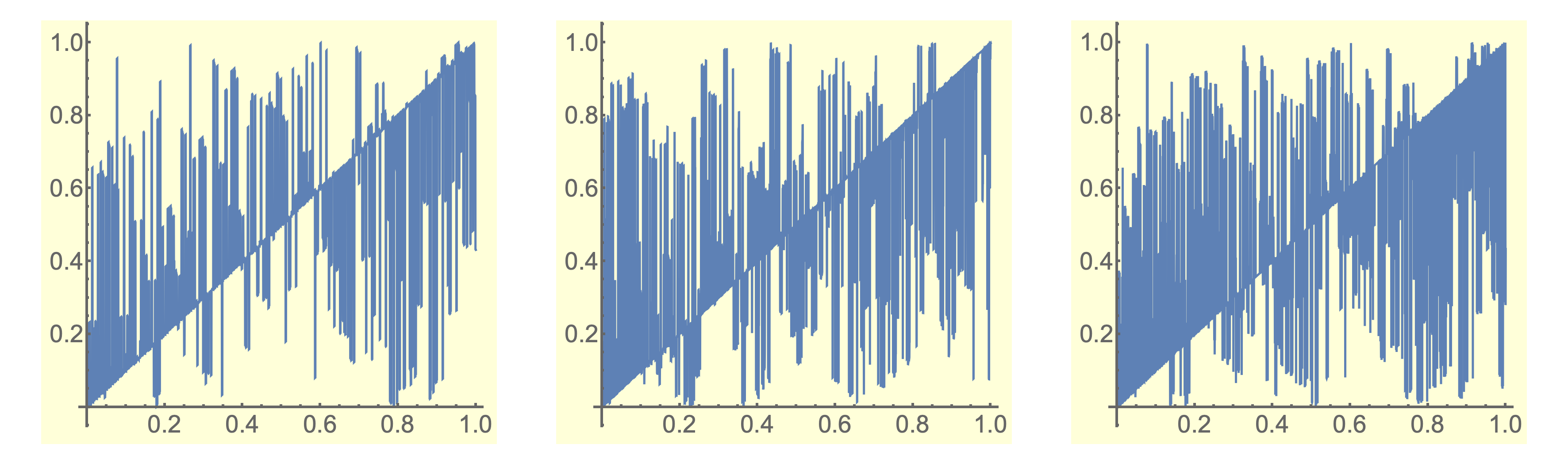}
\caption{The IIET $\iiet_{20}^j$ for the substitution $A\mapsto ABBB, \,  B \mapsto A$, for $j =217, j = 508, $ and $j = 1159$. This substitution has some singular continuous spectrum \cite{Me.Michael.Uwe.Robbie}.}
\label{fig:npfiietbigits}
\end{figure}

\section{Self-similar tilings, S-adic sequences and  fusion tilings of $\R$}
\label{Sec:generalizes-adic}

Our construction works in a number of other situations.  We provide a relatively complete description of the one-dimensional case  below. Extension to higher dimensions could be of interest but we will not discuss that here. 

\subsection{S-adic systems} These systems are generated with substitutions that can vary at each level. A general S-adic system is generated by a {\bf \em directive sequence} of morphisms $\subs_n: \ab_{n} \to \ab_{n-1}^+$ on a sequence $\ab_0,\ab_1, \ab_2, \ab_3, ...$ of finite alphabets. The $n$-supertiles are of the form $\subs_1(\subs_2(...(\subs_n(\alet))...)$, where $\alet \in \ab_n$. A subshift of all admissible sequences is obtained as in definition \ref{def:substitutionsubshift}.
A classic  type of S-adic system is made by taking all substitutions with a given transition matrix $M$ on a single alphabet $\ab$ \cite{ArnouxMizutaniSellami}. In that situation the estimates from corollary \ref{cor:iietest} still hold, and a single dual substitution can be chosen to construct the flow view.

For S-adic symbolic systems the notion of recognizability becomes more nuanced \cite{BertheRecog2018}. We assume the strongest form: {\em full recognizability}, where each substitution in the directive sequence is recognizable.
We also assume that the subshift admitted by the supertiles is minimal. In this case there is no direct Perron-Frobenius theorem to give us information about the natural lengths or frequencies. However, there is a sequence of transition matrices, and any invariant measure for the subshift must obey a type of transition-consistency that forces the left equation of \eqref{eqn:transitionconsistent} to be true at each level.

There are two adaptations needed to our proof. First, we need an address set $\domain_{n}$ for each $\subs_n$, constructed as before \eqref{def:domain}. Addresses for $n$-cylinders will be those elements of $\domain_1\times \domain_2\times ...\times \domain_n$ that represent allowable supertile inclusions. Those inclusions are given at each level as in equation \ref{eqn:parentset} as follows. For each $\alet \in \ab_n$,
\[\locn{\alet} =\{(\blet,j) \in \domain_{n+1} \, | \, \subs_{n+1}(\blet)(j) = \alet\}
=\{(\blet, j) \in \domain_{n+1} \, | \, \alet \text{ is the } j\text{th letter of }\subs_{n+1}(\blet ) \}.\]
These provide the edge sets for the Bratteli diagram. The refining sequence $\{\Partition_n\}$ of $\seqsp$ is canonical, as is the order on the Bratteli diagram. 

The second adaptation is that at every level we need to produce a new left endpoint function $\lft_n: \locn{\alet} \to [0, 1)$ for each $\alet \in \ab_n$. Because the measure is transition-consistent, we can partition the intervals for $n$-supertiles into the subintervals corresponding to their $(n+1)$-tiles and define $\lft_n$ with them as in equation \eqref{eqn:le(aj)}. As before, some initial partition using the alphabet $\ab_0$ on which the sequence space is based is required to make $\mtcong_0$. This time we are unable to write a direct formula and so we use the recursive form of equation \eqref{eqn:LE(p)} to define
\begin{equation}
\Lft_n((\alet_1, j_1), \ldots, (\alet_{n},j_{n})) = \Lft_{n-1}((\alet_1, j_1), \ldots, (\alet_{n-1},j_{n-1})) + \lft_n(\alet_n, j_n). 
\end{equation}
 The rest of the proofs are made with these adaptations.

\subsection{Tilings of $\R$}
Self-similar and fusion tilings of $\R$ are suspensions of substitution and S-adic systems, respectively, and in this case it is the first return map to a transversal that is conjugate to the IIET. The height function for the suspension is given by a vector $\vec{l}$ of lengths for the tiles. We will denote by $\vec{r}$ the right eigenvector for $\lambda$ that is a probability vector. To 
compare two tiling systems given by different length vectors we follow \cite{ClarkSadunSize} and normalize them by requiring that $\vec{l}\cdot\vec{r} = 1$ for each length vector. 

In the next few examples we choose a few substitutions and compare approximations to the flow views of their transversals with unit length versus natural length tiles. We briefly discuss what dynamical features appear to be visible.

\begin{example}
Figure \ref{fig:ARnaturallengthflowview} shows the flow views for the tribonacci substitution ($A\mapsto AB, \,  B \mapsto AC, \, C \mapsto A$) first as unit tiles and then as its suspension using the normalized natural tile lengths. 
\begin{figure}[ht]
\centering
\includegraphics[width=.9\textwidth]{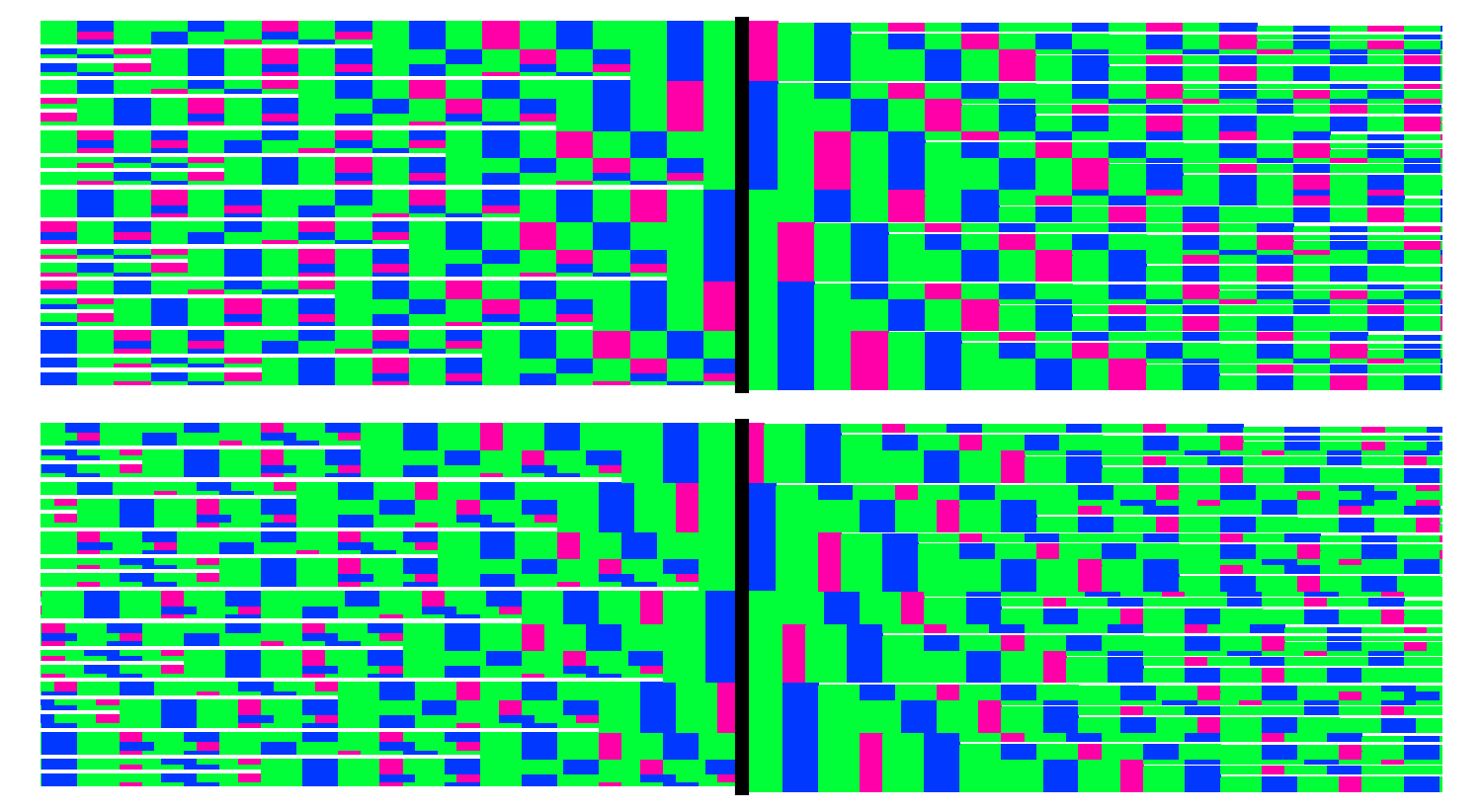}
\caption{Flow view approximations for the tribonacci substitution. Top: unit length tiles. Bottom: natural length tiles.  For the tiling flow, the IIET is the first return map to the transversal of all tilings with an endpoint at 0.}
\label{fig:ARnaturallengthflowview}
\end{figure}

From \cite{ClarkSadunSize} we know that any appropriately scaled suspensions over the tribonacci substitution are topologically conjugate to one another as tiling flows. Inspection of the flow views supports this notion: moving the vertical line at the origin to the right or left, which corresponds to shifting in the tiling space, produces changes that appear to be fairly well controlled and so a conjugacy seems possible. The conjugacy is non-local \cite{ClarkSadunSize}, exemplifying the fact that tiling flows do not have a Curtis-Hedlund-Lyndon theorem. 
\exend
\end{example}

\begin{example}
Figure \ref{fig:naturallengthflowview} shows the flow views for the well-studied non-Pisot substitution rule $A \mapsto ABBB, \,  B \mapsto A$ from figure \ref{fig:npfiietbigits}.
\begin{figure}[ht]
\centering
\includegraphics[width=.9\textwidth]{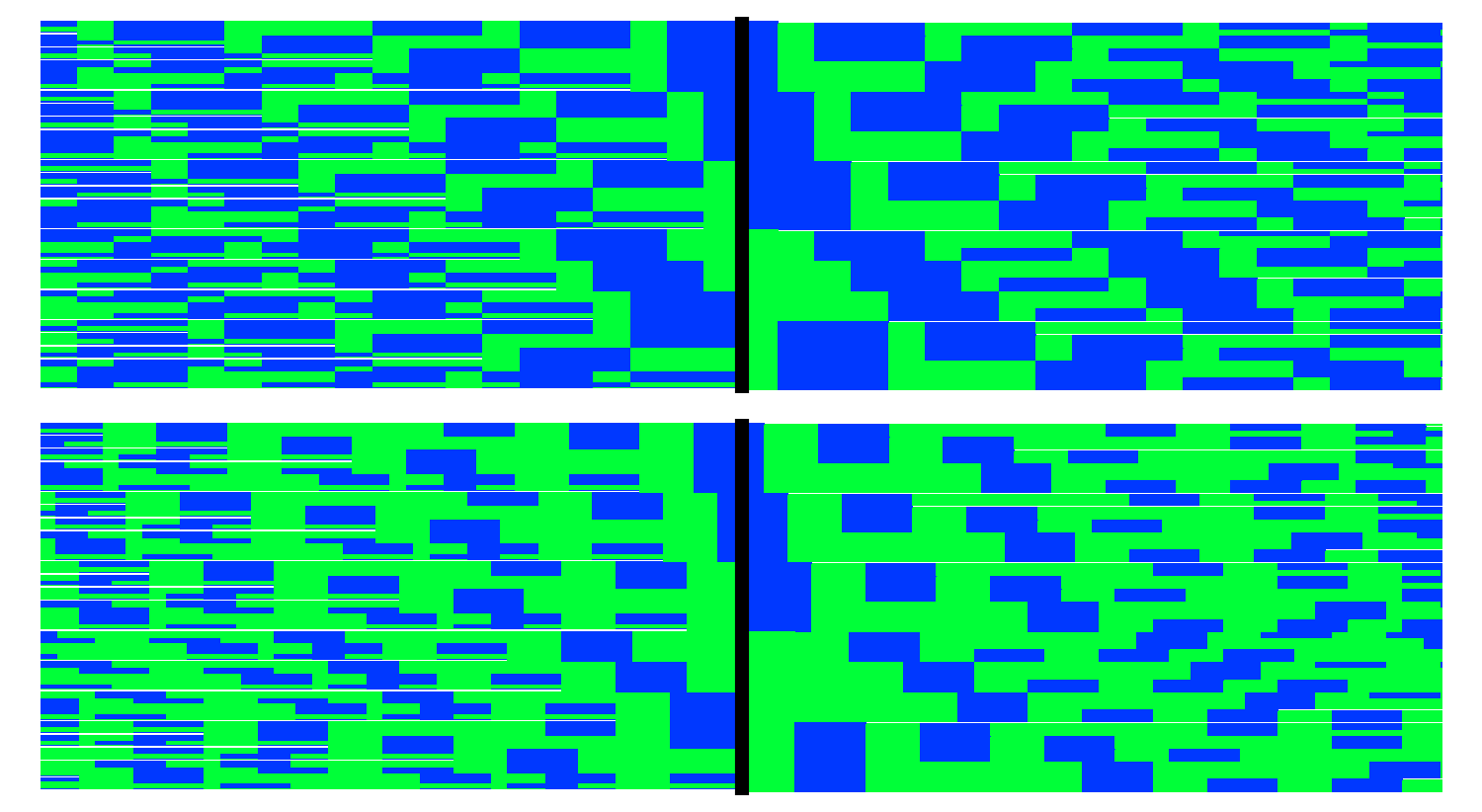}
\caption{Flow view approximations for $A \mapsto ABBB, B \mapsto A$. Top: unit tiles. Bottom: natural length, properly scaled.}
\label{fig:naturallengthflowview}
\end{figure}
This substitution has two distinct eigenvalues of modulus $> 1$. Again from \cite{ClarkSadunSize} we know that even when properly normalized, the tiling flows are not topologically conjugate unless the tile lengths are rationally related. The two flows shown thus represent tiling dynamical systems that are are not topologically conjugate. Again if we imagine translation as moving the vertical bar right or left, we see a sort of distortion that is the obstruction to topological conjugacy.
\exend
\end{example}

\begin{example}[Chacon substitution] The Chacon substitution $A \mapsto AABA$ and $B \mapsto B$ is not primitive, but its subshift is minimal and uniquely ergodic and therefore it has canonical IIETs.
\begin{figure}[ht]
\centering
\includegraphics[width=.9\textwidth]{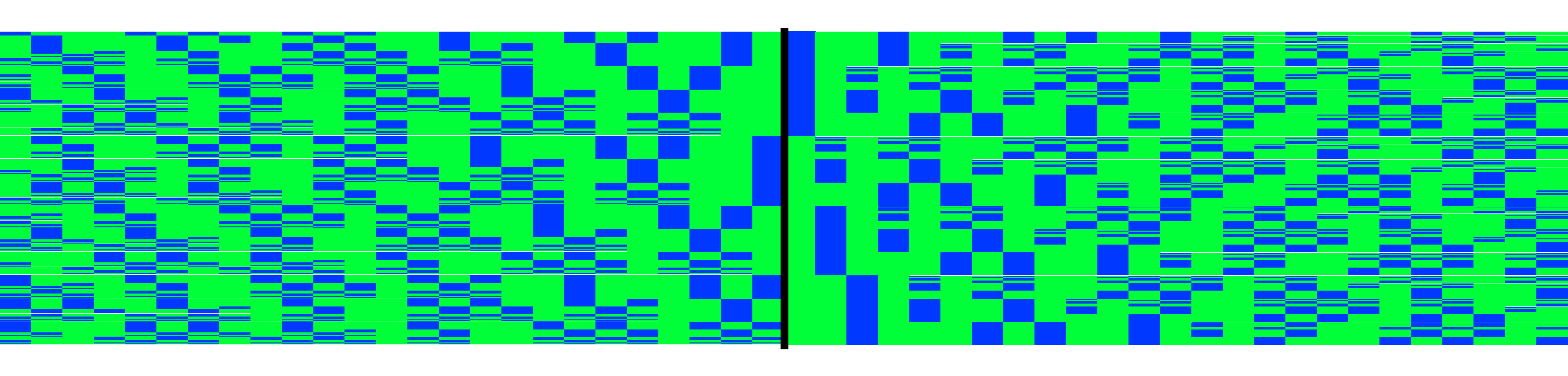}
\caption{There is no meaningful natural length flow view to compare to the Chacon substitution flow view with unit length tiles.}
\label{fig:naturallengthflowviewchac}
\end{figure}
The lack of primitivity prohibits a representative flow view for the natural length tiles because the natural length for ``$B$'' is 0, so the result is a 3-odometer on just the $A$ tile. That flow view would look like a solid green rectangle because all the tiles are green intervals. We show only the unit length flow view in figure \ref{fig:naturallengthflowviewchac}.

The Chacon subshift is weakly mixing but not strongly mixing \cite{ChaconWM}. As we move to the edges of the flow view we see that the organization present at the origin degrades rapidly when compared to the tribonacci or period-doubling flow views.
\exend
\end{example}

\section{Questions, comments, observations}
\label{sec:questions}
The author has fallen down numerous rabbit holes of high (to her) entertainment value. The following is a somewhat random assortment of her favorites.

\smallskip
\noindent{\ding{68}} Any substitution whose subshift is minimal and recognizable has numerous IIET/flow view combinations that represent it. In the simplest case we could consider only the set of canonical IIETs given by all possible dual substitutions. This produces a finite list of $(\mtcong_j,\iiet_j)$s that all represent $(\seqsp,\shift,\mu)$. 
\begin{itemize}
\item What information is contained in the geometry of the maps
$\mtcong_2\circ \mtcong_1^{-1}$ of $[0,1]$?
\item Many substitutions have a {\bf periodic dual substitution} which, if treated as recognizable, would be an odometer. What is the significance of this?
\item Is there an {\bf optimal choice} of dual substitution to represent a system?
\end{itemize}

\noindent{\ding{68}} 
The flow view construction provides a link between all substitutions and S-adic systems that share a common primitive transition matrix $M$.
\begin{itemize}
\item For any given substitution, an S-adic system using any or all of the choices of dual substitution could be used, in which case the representation would be non-canonical. 
\item Alternatively, S-adic systems in which the directive sequence shares a transition matrix can be modeled with a dual substitution for its flow view.
\item Take two subshifts with coordinate maps given by the same dual substitution. This gives a measure-theoretic bijection between the subshifts. How do the properties of the matrix affect the properties of this bijection?
\end{itemize}

\smallskip
\noindent{\ding{68}} The Fibonacci substitution subshift is known to correspond to an {\bf exchange of two intervals}, but the $\iiet$ that appears in figure \ref{IIETIntro} of our introduction  exchanges infinitely many. That's because there are only two choices of dual substitution for $\fib$, and both yield exchanges of infinitely many intervals. There are two main ways to enlarge the scope. One is to allow all canonical $(\mtcong_j,\iiet_j)$s for all powers of the substitution. The other is to allow the subdivisions to vary at each stage, producing an S-adic structure of dual substitutions.  The latter has not been investigated so far, but the former allows the Fibonacci two-interval exchange to appear from our process. With $\fib^2$ and the correct  choice of dual substitution, figure \ref{fig:ffiietapprox} shows how the approximants become the two-interval exchange.
\begin{figure}[ht]
\centering
\includegraphics[width=.9\textwidth]{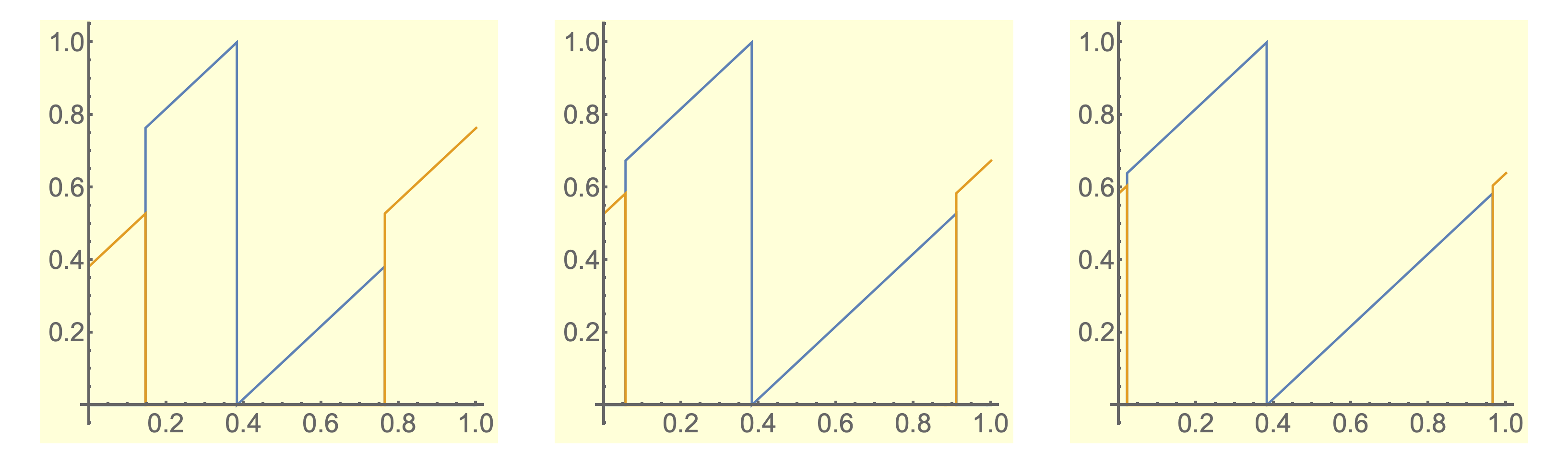}
\caption{The first three approximants for an IIET of $\fib^2$.}
\label{fig:ffiietapprox}
\end{figure}
\begin{itemize}
\item {\sc Conjecture:} If a substitution subshift is measurably conjugate to a {\bf finite} interval exchange transformation, that transformation is a member of the family of canonical infinite interval exchange transformations of its powers. 
\item 
One can define the {\bf \em efficiency} of an IIET representation in terms of how many total translations/intervals are possible given the size of the substitution versus how many the IIET actually needs. For instance, since the Fibonacci subshift can be represented with an IET of two intervals, the efficiency of any other canonical representation can be compared.
\end{itemize}

\smallskip
\noindent{\ding{68}} The destination under $\mtcong$ of a randomly selected tiling is `expected' to be
$\int_\seqsp \mtcong(\T) d\mu = \int_0^1 x \, dx = 1/2$.  This indicates that elements of $\mtcong^{-1}(1/2)$ can be considered the most average or representative elements of $\seqsp$. The stability and other properties of this set under choice of dual substitution is not yet known.

\smallskip
\noindent{\ding{68}} In \cite{OlgaHenk}, the authors begin with the dyadic odometer as an IIET and then `twist' it, obtaining new translation surfaces encoding, eventually, S-adic or substitution subshifts as minimal components. Remarkably, a power of the Fibonacci substitution appears as a minimal component and does not have the odometer as a factor \cite[Section 5.4]{OlgaHenk}. 
\begin{itemize}
\item The method is a sort of inverse of the one in this paper, beginning with an IIET and ending with a subshift. To what degree can the two together classify the IIETs of self-inducing subshifts?
\item Can our method `find' the dyadic odometer for the Fibonacci substitution? If so can it find other, perhaps Fibonacci-number-adic, odometers?
\end{itemize}

\bibliography{master}{}

\bibliographystyle{plain}

\end{document}